\documentclass{lms}

\usepackage[centertags]{amsmath}
\usepackage{amssymb}
\usepackage[english]{babel}
\usepackage[all]{xy}
\usepackage[active]{srcltx}
\usepackage{hyperref}
\usepackage{graphicx}

\title[Higher Whitehead torsion]{Higher Whitehead torsion and the geometric assembly map}
\author{Wolfgang Steimle}

\classno{19D10, 19J10, 55R10}

\DeclareMathOperator*{\colim}{colim}
\DeclareMathOperator*{\hocolim}{hocolim}
\DeclareMathOperator*{\holim}{holim}

\begin{document}

\newtheorem{thm}{Theorem}[section]
\newtheorem{cor}[thm]{Corollary}
\newtheorem{lem}[thm]{Lemma}
\newtheorem{prop}[thm]{Proposition}

\newnumbered{defn}[thm]{Definition}
\newnumbered{rem}[thm]{Remark}
\newnumbered{obs}[thm]{Observation}
\newunnumbered{notat}{Notation}


\newcommand{\lto}{\longrightarrow}
\newcommand{\mor}{\opn{mor}}
\newcommand{\pathto}{\rightsquigarrow}
\newcommand{\opn}{\operatorname}
\newcommand{\id}{\opn{id}}
\newcommand{\Proj}{\opn{Proj}}
\newcommand{\Cyl}{\opn{Cyl}}
\newcommand{\Sing}{\opn{Sing}}
\newcommand{\cone}{\opn{cone}}
\newcommand{\cyl}{\opn{cyl}}
\newcommand{\Wh}{\mathrm{Wh}}
\newcommand{\Whp}[1]{\Wh(\pi #1)}
\newcommand{\WhPL}{\Wh^{PL}}
\newcommand{\G}{\mathrm{G}}
\newcommand{\PL}{\mathrm{PL}}
\newcommand{\TOP}{\mathrm{TOP}}
\newcommand{\DIFF}{\mathrm{DIFF}}
\newcommand{\Homeo}{\mathrm{Homeo}}
\newcommand{\hofib}{\mathrm{hofib}}
\newcommand{\FIB}{\opn{FIB}}
\newcommand{\Aut}{\opn{Aut}}
\newcommand{\Out}{\opn{Out}}
\newcommand{\map}{\opn{map}}
\newcommand{\Hom}{\opn{Hom}}
\newcommand{\End}{\opn{End}}
\newcommand{\GL}{\opn{GL}}
\newcommand{\im}{\opn{im}}
\newcommand{\coker}{\opn{coker}}
\newcommand{\simp}{\opn{simp}}
\newcommand{\res}{\opn{res}}
\newcommand{\ind}{\opn{ind}}
\newcommand{\eval}{\opn{eval}}
\newcommand{\xican}{\xi_{\opn{can}}}
\newcommand{\Wall}{\mathrm{Wall}}
\newcommand{\sr}{\opn{sr}}
\newcommand{\Ho}{\opn{Ho}}

\newcommand{\arincl}{\ar@{^{(}->}}
\newcommand{\arinclinv}{\ar@{_{(}->}}
\newcommand{\norm}[1]{\Vert #1\Vert}
\newcommand{\eqclass}[1]{#1/{\sim}}
\newcommand{\sections}[2]{\Gamma\biggl(\begin{array}{c} {#1}\\ \downarrow\\ {#2}\end{array} \biggr)}
\newcommand{\verticalhofib}[3]{\hofib_{#1}\biggl(\begin{array}{c} {#2}\\ \downarrow\\ {#3}\end{array} \biggr)}

\newcommand{\lift}[4]{\opn{Lift}\left(\renewcommand{\arraystretch}{0.5}\begin{matrix} && {#1} \\ && \downarrow \\ {#2} & \overset{#3}{\longrightarrow} & {#4} \end{matrix}\right)\renewcommand{\arraystretch}{1.0}} 
\newcommand{\smalllift}[4]{\opn{Lift}\left(\begin{smallmatrix} && {#1} \\ && \downarrow \\ {#2} & \overset{#3}{\longrightarrow} & {#4} \end{smallmatrix}\right)}

\newcommand{\Ab}{\mathrm{Ab}}
\newcommand{\RMod}{R\mathrm{-mod}}
\newcommand{\Sp}{\mathrm{Spectra}}
\newcommand{\OSp}{\Omega\mathrm{-Spectra}}
\newcommand{\Cat}{\mathrm{Cat}}
\newcommand{\op}{^{op}}
\newcommand{\ltwotop}{\ell^2\textendash \mathrm{Top}}
\newcommand{\cpCW}{\mathbf{cpCW}}
\newcommand{\cat}{\mathbf{cat}}
\newcommand{\CGHaus}{\mathbf{CGHaus}}
\newcommand{\sk}{\opn{sk}}

\newcommand{\RR}{\mathbb{R}}
\newcommand{\ZZ}{\mathbb{Z}}
\newcommand{\NN}{\mathbb{N}}
\newcommand{\QQ}{\mathbb{Q}}
\newcommand{\CC}{\mathbb{C}}
\newcommand{\calB}{\mathcal{B}}
\newcommand{\calC}{\mathcal{C}}
\newcommand{\calE}{\mathcal{E}}
\newcommand{\calH}{\mathcal{H}}
\newcommand{\calL}{\mathcal{L}}
\newcommand{\calP}{\mathcal{P}}
\newcommand{\calQ}{\mathcal{Q}}
\newcommand{\calS}{\mathcal{S}}
\newcommand{\calR}{\mathcal{R}}
\newcommand{\calU}{\mathcal{U}}
\newcommand{\Bun}{\mathrm{Bun}}
\newcommand{\Fib}{\mathrm{Fib}}
\newcommand{\fib}{\opn{fib}}
\newcommand{\Rfd}{\mathcal{R}^{fd}}
\newcommand{\Rf}{\mathcal{R}^f}
\newcommand{\Rfh}{\mathcal{R}^f_h}


\bibliographystyle{alpha}

\SelectTips{eu}{10}

\hyphenation{pa-ram-e-trized}

\extraline{This work has been supported by the Graduiertenkolleg ``Analytische Topo\-logie und Metageometrie'' of the Deutsche Forschungsgemeinschaft, by the Hausdorff Center of Mathematics and by the Leibniz-Preis of Wolfgang L\"uck.}


\maketitle

\begin{abstract}
We construct a higher Whitehead torsion map, using algebraic $K$-theory of spaces, and show that it satisfies the usual properties of the classical Whitehead torsion. This is used to describe a ``geometric assembly map'' defined on stabilized structure spaces in purely homotopy theoretic terms.
\end{abstract}

\section{Introduction}

\noindent Given a space $X$, the structure space $\calS_n(X)$ is the space of all homotopy equivalences $M\to X$ where $M$ is an $n$-dimensional manifold (of some given type, such as compact or closed, differentiable, PL, or topological). Obviously $X$ has the homotopy type of such a manifold if and only if the structure space of $X$ is non-empty. In this case, $\pi_0\calS_n(X)$ is the central object of interest for the classification of manifolds in the homotopy type of $X$; the higher homotopy type of $\calS_n(X)$ is closely related to the study of automorphisms of these and to the classification of families of manifolds homotopy equivalent to $X$ \cite{Weiss-Williams(2001)}. 

If $p\colon E\to B$ is a given fibration, then the above construction can be generalized as follows: A point in the structure space $\calS_n(p)$ of $p$ is given by a bundle of $n$-dimensional manifolds $E'\to B$ over $B$ together with a fiber homotopy equivalence $E'\to E$. (So the structure space of $X$ as above is the structure space of the trivial fibration $X\to *$.) Pull-back defines a pairing
\[\calS_n(B)\times\calS_k(p)\to\calS_{n+k}(E);\]
thus if $B$ comes with a given structure, then evaluation of this pairing defines a map
\[\alpha\colon \calS_k(p)\to\calS_{n+k}(E).\]
Geometrically this map assembles all the manifold structures on the individual fibers to one manifold structure on $E$, so we call it the \emph{geometric assembly map.} It is relevant for instance in the study of fibering questions.

Algebraic $K$-theory of spaces \cite{Waldhausen(1985)} is a key tool for the understanding of families of manifolds. The connection to the topology of manifolds is established by the stable parametrized $h$-cobordism theorem \cite{Waldhausen-Jahren-Rognes(2008)}, which classifies families of $h$-cobordisms in terms of $K$-theory, in a stable range. Recently, Hoehn \cite{Hoehn(2009)} has used the stable parametrized $h$-cobordism theorem to describe the homotopy type of the stabilized structure space
\[\calS_\infty(p):=\colim\bigl(\calS_n(p)\xrightarrow{\times [0,1]} \calS_{n+1}(p) \xrightarrow{\times [0,1]} \dots \bigr),\]
working with compact topological manifolds (possibly with boundary). More precisely, using works of Dwyer-Weiss-Williams \cite{Dwyer-Weiss-Williams(2003)}, she constructed a specific $K$-theoretic invariant
\[\calS_\infty(p)\to R(p)\]
and showed that this invariant, together with a map $T^v$ that describes tangent bundle information, defines a homotopy equivalence
\[\calS_\infty(p)\xrightarrow{\simeq} R(p)\times\map(E,B\TOP).\]

Hence, there is (up to homotopy) a unique dotted lift making the following diagram homotopy commutative:
\begin{equation}\label{intro:main_diagram}
\xymatrix{
{\calS_\infty(p)} \ar[d]^\alpha \ar[rr]^(.4){\simeq} && R(p)\times\map(E,B\TOP) \ar@{.>}[d]\\
{\calS_\infty(E)} \ar[rr]^(.4)\simeq && R(E)\times\map(E,B\TOP)
}
\end{equation}
The main result of this paper is to give a description of this dotted map purely in terms of homotopy theory.

The first step towards such a description is to observe that if $p$ is a fiber bundle of compact manifolds, then the space $R(p)$ is homotopy equivalent to a specific infinite loop space $\bar R(p)$. A further advantage of $\bar R(p)$ is its better functoriality: Apart from the fact that a map $f\colon B'\to B$ induces a restriction map $f^* \colon \bar R(p)\to \bar R(f^*p)$, any fiberwise map $g\colon p\to p'$ induces a map $g_* \colon \bar R(p)\to \bar R(p')$.

A more formal reasoning shows that Hoehn's map can be replaced by a map
\[\tau\colon \calS_\infty(p)\to \bar R(p)\]
which we call parametrized Whitehead torsion. In fact, if $B$ is a point, then $\pi_0 \bar R(E)\cong\Wh(\pi_1 E)$ and the map
\[\pi_0(\tau)\colon\pi_0\calS_\infty(E)\to\Wh(\pi_1 E)\]
sends the class of a homotopy equivalence $\varphi\colon M\to E$ to its classical Whitehead torsion. Moreover, we will prove:

\begin{thm}
The composition rule, additivity, product rule, and homeomorphism invariance of the classical Whitehead torsion generalize to the parametrized Whitehead torsion.
\end{thm}

See sections \ref{whtors} to \ref{section:additivity} for more precise statements and proofs. -- Coming back to diagram \eqref{intro:main_diagram}, assume (for simplicity) that $B$ is connected and choose some base point $b$ in $B$. Letting $i\colon F:=p^{-1}(b)\to E$ be the inclusion and denoting by $\chi_e(B)\in\ZZ$ the Euler characteristic, we may define
\[\beta\colon \bar R(p)\xrightarrow{\mathrm{restr.}} \bar R(F) \xrightarrow{\chi_e(B)\cdot i_*} \bar R(E).\]\
Using the above constructions and results, it follows:

\begin{thm}\label{intro:main_theorem}
The following diagram commutes up to homotopy:
\[\xymatrix{
{\calS_\infty(p)} \ar[d]^\alpha \ar[rr]^(.4){(\tau, T^v)}_(.4){\simeq} && \bar R(p)\times\map(E,B\TOP) \ar[d]^{\beta\times(+p^*TB)}\\
{\calS_\infty(E)} \ar[rr]^(.4){(\tau, T^v)}_(.4)\simeq && \bar R(E)\times\map(E,B\TOP)
}\]
Here the map $(+p^*TB)$ is given by Whitney sum with the tangent bundle of $B$, pulled back to $E$ via $p$.
\end{thm}

In a sequel paper \cite{Steimle(2011b)}, the author will apply this result to the fibering problem which asks whether a given map between manifolds is homotopic to the projection map of a fiber bundle. Theorem \ref{intro:main_theorem} together with the ``converse Riemann-Roch theorem'' of \cite{Dwyer-Weiss-Williams(2003)} allow to set up a complete obstruction theory to the stable fibering problem.

\subsection*{Organization of the work}

The first three sections of this text are devoted to the definition of the parametrized Whitehead torsion map. It is built upon the works of Dwyer-Weiss-Williams on the parametrized $A$-theory characteristic \cite{Dwyer-Weiss-Williams(2003)}. However, this characteristic is defined for a single bundle, whereas we want to get an invariant on the structure space. Following Hoehn's method in spirit, we therefore consider the parametrized characteristic in a universal situation. So, in the first section, a general method by Hughes-Taylor-Williams \cite{Hughes-Taylor-Williams(1990)} will be presented and slightly simplified, which interprets these structure spaces as spaces of lifts of certain classifying spaces. 

For the convenience of the reader we recall in section \ref{section:DWW} the main results of \cite{Dwyer-Weiss-Williams(2003)} that we need. The actual definition of the parametrized Whitehead torsion map is done in section \ref{whtors}. It is also in this section where some elementary properties will be proved, such as the composition rule and homeomorphism invariance. 

The following two sections are devoted to additivity and the product rule. Section \ref{section:sttors} defines the torsion on the stabilized structure space. In section \ref{fbs} the geometric assembly map is defined and Theorem \ref{intro:main_theorem} is proved. Finally, section \ref{section:h_cob} discusses the question which elements of $\pi_*\bar R(p)$ can be realized by glueing fiberwise $h$-cobordisms.

An appendix collects some technical results on fibrations, which are needed to make the classifying-space machinery work. A second appendix gives the link to classical simple homotopy theory.

\subsection*{Acknowledgements}

This work is part of the author's PhD thesis written at the University of M\"unster under the supervision of Wolfgang L\"uck. I am grateful to him for his support and encouragement as well as to Bruce Hughes and Bruce Williams for many helpful discussions. I am also thankful to the referee for his or her careful reading and valuable comments; in particular the simplified proof of Lemma \ref{strspaces:lax_naturality_of_excisive_characteristic} is due to him or her.


\section{Structure spaces on fibrations}\label{strspaces}

The goal of this section is to review the definition of the space $\calS_n(p)$ of $n$-dimensional compact manifold structures on a fibration $p\colon E\to B$ and to show that under suitable assumptions, this structure space is weakly homotopy equivalent to the space of lifts in the following diagram
\begin{equation}\label{strspaces:diagram_of_lifts}
\xymatrix{
 && {\coprod_{[M]}B\TOP(M)} \ar[d]\\
B \ar[rr]^p \ar@{.>}[rru] && BG(F)
}
\end{equation}
Here $\TOP(M)$ is the simplicial group of self-homeomorphisms of $M$, $G(M)$ is the simplicial monoid of self-homotopy equivalences, and the coproduct ranges over the isomorphism classes of compact $n$-man\-ifolds homotopy equivalent to $F$. This result is well-known; the goal is to present (and slightly simplify) the machinery of \cite{Hughes-Taylor-Williams(1990)} which can be used for a proof, as analogous results will be needed in a variety of similar situations later on.

We will always work in the category $\CGHaus$ of compactly generated Hausdorff spaces, and all constructions are to be taken in that category.

Fix once and for all a set $\calU$ of
cardinality at least $2^{\vert\RR\vert}$ and a bijection $\calU\times\calU\to\calU$.

Let $p\colon E\to B$ be a fibration over a paracompact space $B$. By definition, an
\emph{$n$-dimensional compact manifold structure on $p$} is a commutative diagram
\[\xymatrix{
E' \ar[rd]_{p'} \ar[rr]^\varphi&& E\ar[ld]^p \\
&B,
}\]
where 
\begin{itemize}
\item $p'$ is a bundle of $n$-dimensional compact (not necessarily closed)
manifolds (i.e.~$p'$ is the projection map of a fiber bundle with fibers $n$-dimensional compact topological
manifolds), 
\item  $\varphi$ is a homotopy equivalence, and
\item $E'$ is a subset of $B\times \calU$ such that $p'$ corresponds to the projection from $B\times\calU$ to $B$.
\end{itemize}

The last condition ensures that all $n$-dimensional compact manifold structures on $p$ form a
set $\calS_n(p)_0$. Given a pull-back square of fibrations
\begin{equation}\label{strspaces:pull_back_square_of_fibrations}
\xymatrix{
E_0 \ar[rr] \ar[d]^{p_0} && E\ar[d]^p \\
B_0 \ar[rr]^f && B
}
\end{equation}
a compact manifold structure $x=(p',E')$ on $p$ induces a compact manifold structure $f^*x=(f^*p', f^*E')$ on
$p_0$ of the same dimension by pull-back via $f$. (Notice that $f^*E'$ is naturally a subset of $B_0\times \calU$. This ensures that the pull-back operation is strictly functorial.)

One can therefore define a simplicial set $\calS_n(p)_\bullet$ by
\[\calS_n(p)_k:=\calS_n(p\times\id_{\Delta^k})_0,\]
such that the simplicial operations are induced by the pull-back operation on the level of standard simplices.

\begin{defn}
The \emph{space of $n$-dimensional compact manifold structures on $p$} is the
geometric realization
\[\calS_n(p):=\vert\calS_n(p)_\bullet\vert.\]
If $B$ is a point, we simply write $\calS_n(E)$ for $\calS_n(p)$.
\end{defn}

Given a fiber homotopy equivalence $\psi\colon p'\to p$ of fibrations over $B$,
we clearly obtain a simplicial map $\psi_*\colon\calS_n(p')_\bullet \to 
\calS_n(p)_\bullet$ by composition, inducing a map on the structure spaces. On the
other hand, given a pull-back square \eqref{strspaces:pull_back_square_of_fibrations} of fibrations, the restriction
operation leads to a map $f^*\colon\calS_n(p)\to\calS_n(p_0)$.

The following two lemmas (which are well-known and stated for completeness) show that both the covariant and the contravariant operation satisfy homotopy invariance.

\begin{lem}\label{strspaces:homotopy_invariance_I}
If $\psi,\varphi\colon p'\to p$ are fiber homotopy equivalences that are fiber
homotopic, then $\psi_*\simeq\varphi_*\colon\calS_n(p')\to\calS_n(p)$.
\end{lem}

\begin{lem}
\begin{enumerate}
\item Let $i_0,i_1\colon B\to B\times I$ be the inclusions at 0 and 1. Then $i_0^*,
i_1^*\colon\calS_n(p\times\id_I)\to\calS_n(p)$ are homotopic.
\item If in \eqref{strspaces:pull_back_square_of_fibrations}, $f$ is a homotopy equivalence, then so is $f^*\colon\calS_n(p)\to\calS_n(p_0)$.
\end{enumerate}
\end{lem}

We now proceed to show that under suitable hypotheses the structure space is weakly homotopy equivalent to the space of lifts in \eqref{strspaces:diagram_of_lifts}.

For two spaces $F$ and $B$, with $B$ paracompact, let $\mathbf{Bun}_n(B;F)$ be
the
category where an object is a bundle $E\to B$ with fibers compact
$n$-dimensional topological manifolds homotopy equivalent to $F$. Again we
assume that $E$, as a set, is a subset of $B\times \calU$ and the inclusion map
$E\to B\times\calU$ is a map over $B$. Morphisms in this category are to be
isomorphisms of such bundles. 

Denote by $\cpCW$ the category of compact
CW spaces, with continuous maps. Then the rule $X\mapsto \mathbf{Bun}_n(B\times
X;F)$ defines a functor 
\[\mathbf{Bun}_n(B;F)\colon\cpCW\op\to\cat,\] 
to the category of small categories. By giving an explicit system of simplices
in $\cpCW$ (i.e.~an embedding of categories $\Delta\to\cpCW$ such that $[n]$
maps to an $n$-simplex and a morphism $[m]\to[n]$ maps to the corresponding face
or degeneracy map), we can precompose to get a simplicial small category
(i.e.~simplicial object in the category of small
categories) 
\[\mathbf{Bun}_n(B;F)_\bullet\colon \Delta\op\to\cpCW\op\to\cat.\]

It is obvious that different choices of systems of simplices lead to 
naturally isomorphic simplicial small categories.

Similarly define $\mathbf{Fib}(B;F)$ to be the category
where an object is a (Hure\-wicz) fibration over $B$ with fibers homotopy
equivalent to $F$. We also require that the total space of the fibration is a subset of $B\times\calU$ such that the inclusion map is fiberwise over $B$. A morphism in $\mathbf{Fib}(B;F)$ is to be
a fiber homotopy equivalence. Again this gives rise to a functor
$\cpCW\op\to\cat$ by the rule $X\mapsto \mathbf{Fib}(B\times X;F)$ and
therefore to a simplicial category, by precomposing with the system of simplices. 
Since $B$ is supposed to be paracompact, any bundle over $B\times
\Delta^n$ is a fibration, such that we obtain a natural transformation
$\mathbf{Bun}_n(B;F)_\bullet\to\mathbf{Fib}(B;F)_\bullet$.

Consider more generally any functor $\mathbf C\colon\cpCW\op\to\cat$. Here are three
properties that such a functor may have. In fact all our examples of such
functors will satisfy all of these properties. It is useful to think of an
object of $\mathbf C(X)$ as an object ``over'' $X$ and the functoriality
operation as a restriction.

\begin{description}
\item[Amalgamation property] For any push-out square
\[\xymatrix{
X_0 \arincl[rr] \arinclinv[d]&& X_2 \ar[d]\\
X_1 \ar[rr] && X
}\]
of compact CW spaces such that for $i=1,2$, the map $X_0\to X_i$ is the inclusion of a subcomplex,  the square
\[\xymatrix{
{\mathbf C(X)} \ar[rr] \ar[d] && {\mathbf C(X_2)} \ar[d]\\
{\mathbf C(X_1)} \ar[rr]&& {\mathbf C(X_0)}
}\]
with inclusion-induced maps is a pull-back.
\end{description}

\begin{rem}
\begin{enumerate}
\item In comparison to \cite{Hughes-Taylor-Williams(1990)} this condition is slightly different. This difference does not affect the conclusions we are going to draw.
\item To verify that a commutative square 
\[\xymatrix{
{\mathbf A} \ar[rr] \ar[d] && {\mathbf B} \ar[d]\\
{\mathbf C} \ar[rr] && {\mathbf D}
}\]
of categories is a pull-back, it is enough to
verify the following two assertions:
\begin{enumerate}
\item Given any two objects $b\in\mathbf B$
and $c\in\mathbf C$ projecting to the same element $d\in\mathbf D$, there exists
a unique $a\in\mathbf A$ projecting to $b$ and $c$.
\item Given any two morphisms $\beta$ in $\mathbf B$ and $\gamma$ in $\mathbf
C$ projecting to the same morphism $\delta$ in $\mathbf D$, there exists a
unique morphism $\alpha$ in $\mathbf A$ projecting to $\beta$ and $\gamma$.
\end{enumerate}
\end{enumerate}
\end{rem}

\begin{description}
\item[Straightening Property] Denote by $p\colon\Delta^k\times I\to\Delta^k$
the projection and by $i\colon \Delta^k\to\Delta^k\times I$ the inclusion at 0.
For any object $E\in\mathbf C(\Delta^k\times I)$, there are to be morphisms
\[F\colon E\to p^*i^* E \quad\mathrm{and} \quad G\colon p^*i^* E\to E\]
in $\mathbf C(\Delta^k\times I)$ which are both the identity map $i^*E\to i^*E$
upon restriction with $i$.
\item[Fill-in property] Again denote by $p\colon\Delta^k\times I\to\Delta^k$
the projection. For any diagram
\[E_0 \xrightarrow{\varphi_0} E \xleftarrow{\varphi_1} E_1\]
in $\mathbf C(\Delta^k)$, there is an object $\bar E$ over $\Delta^k\times I$ and a morphism $\Phi\colon
\bar E\to p^* E$ which restricts to $\varphi_0$ over 0 and to $\varphi_1$ over
1. 

Moreover, given a second diagram
\[F_0 \xrightarrow{\varphi_0} F \xleftarrow{\varphi_1} F_1\]
in $\mathbf C(\Delta^k)$ together with morphisms
$\psi_i\colon F_i\to E$, $i=0,1$, which agree with the data $(E_0, E_1,
\varphi_0, \varphi_1)$ when restricted to a collection of faces of $\Delta^k$,
there are extensions $(\bar E, \Phi)$ and $(\bar F, \Psi)$ of $(E_0, E_1,
\varphi_0, \varphi_1)$ and $(F_0, F_1, \psi_0, \psi_1)$ that agree when
restricted to the same collection $\times I$.
\end{description}

We again restrict our functor $\mathbf C\colon\cpCW\op\to\cat$ to a
simplicial category $\mathbf C_\bullet\colon\Delta\op\to\cat$. Any simplicial
small category $\mathbf C_\bullet$ gives rise to the following three simplicial sets:
\begin{itemize}
\item The 0-nerve $C_\bullet:=N_0\mathbf C_\bullet$,
\item The (diagonal of the) bisimplicial set $N_\bullet\mathbf C_\bullet$, and
\item For each object $c\in\mathbf C_0$, the (diagonal of the) nerve
$N_\bullet\End(c)_\bullet$ of the simplicial monoid $\End(c)_\bullet$. The
$k$-simplices of $\End(c)_\bullet$ are just the endomorphisms of $c\in\mathbf
C_k$. (The object $c$ is understood to be lifted to $\mathbf C_k$ via the
degeneracy operation.)
\end{itemize}

If the original functor $\mathbf C$ satisfies the Amalgamation, Straightening,
and Fill-in properties, then the following holds (see \cite[\S\S~7 and
8]{Hughes-Taylor-Williams(1990)}):
\begin{enumerate}
\item All simplicial sets $N_k\mathbf C_\bullet$ are Kan.
\item The natural inclusion $C_\bullet=N_0\mathbf C_\bullet\to N_\bullet\mathbf
C_\bullet$ is a homotopy equivalence.
\item The natural inclusion $\End(c)_k\to \mathbf C_k$, for all objects
$c\in\mathbf C_0$, gives rise to a homotopy equivalence
\[\coprod_{[c]\in\pi_0 C_\bullet} N_\bullet\End(c)_\bullet \to N_\bullet\mathbf
C_\bullet.\]
\item Suppose that $\mathbf D\colon\cpCW\op\to\cat$ is another functor
and that $\mathbf f_\bullet\colon\mathbf D_\bullet\to \mathbf C_\bullet$ is a natural
transformation between the associated simplicial categories. For $c\in\mathbf
C_0$, the comma categories $\mathbf f_k / c$ define a simplicial small category
$\mathbf f_\bullet / c$ whose zero-nerve fits into a homotopy fibration sequence
\[ N_0 \mathbf f_\bullet / c \to D_\bullet \overset{f_\bullet}{\to} C_\bullet\]
(homotopy fiber over the point $c\in C_0$).
\end{enumerate}

The following has been proven in a slightly different form in \cite[\S
7]{Hughes-Taylor-Williams(1990)}. Recall that a space $B$ is called \emph{ULC} (or \emph{locally equiconnected}) if there is a neighborhood $U\subset B\times B$ of the diagonal and a homotopy
\[H\colon U\times I\to B\]
between the first and the second projection which is relative to the diagonal. For example, if $B$ is a metrizable ANR (e.g., a locally finite CW complex), then it is also ULC.

\begin{lem}
Both functors $\mathbf{Bun}_n(B;F)$ and $\mathbf{Fib}(B;F)$ satisfy the
Amalgamation, Straightening, and Fill-in properties provided $B$ is  
metrizable and ULC.
\end{lem}

\begin{proof}
For bundles, the Amalgamation property is classical, and so is the
Straightening property. Fill-ins are given by mapping cylinders. (Strictly
speaking, given a map $\varphi\colon E\to
E'$ over $B$, think of its mapping cylinder as a subset
$E\times [0,1)\cup E'\times \{1\}$ of $B\times I\times \calU$, endowed with
the suitable topology.)

As for fibrations, see Proposition \ref{fibr:glueing_over_base_spaces} for the Amalgamation property. Straightening for fibrations follows from homotopy lifting. The existence of fill-ins for fibrations is proven in Proposition \ref{fibr:fill_in}.
\end{proof}

\begin{cor}\label{strspaces:structure_space_as_homotopy_fiber}
\begin{enumerate}
\item For any fibration $p\colon E\to B$ over a metrizable ULC base space which also is an object of
$\mathbf{Fib}(B;F)$, there is a simplicial
homotopy equivalence 
\[\calS_n(p)_\bullet \to \hofib_p(\Bun_n(B;F)_\bullet\to\Fib(B;F)_\bullet)\]
which is natural in $B$. 
\item We have
\[\Bun_n(*;F)_\bullet\simeq \coprod_{[M]} B\TOP(M)_\bullet,\quad
\Fib(*;F)_\bullet\simeq BG(F)_\bullet,\]
where $\TOP(M)_\bullet$ resp.~$G(F)_\bullet$ denotes the simplicial monoid of
self-homeo\-morphisms resp.~self-homotopy equivalences, and the coproduct ranges
over all isomorphism classes of compact $n$-manifolds homotopy equivalent to
$F$.
\end{enumerate}
\end{cor}

Now we are going to show that we have obtained suitable models for classifying
spaces.

\begin{lem} \label{strspaces:disassembly_map}
Let $X_\bullet$ be a locally finite simplicial set.
\begin{enumerate}
\item There are natural simplicial isomorphisms
\begin{align*}
\Bun_n(\vert X_\bullet\vert;F)_\bullet &\cong
\map_\bullet(X_\bullet,\Bun_n(*;F)_\bullet)\\
\Fib(\vert X_\bullet\vert;F)_\bullet &\cong \map_\bullet(X_\bullet,\Fib(*;F)_\bullet)
\end{align*}
\item For any fibration $p\colon E\to B$, with $B$ the geometric realization of
a locally finite simplicial set $B_\bullet$, there is a natural simplicial
isomorphism
\[\hofib_p(\Bun_n(B;F)_\bullet\to\Fib(B;F)_\bullet) \to
\lift{\Bun_n(*;F)_\bullet}{B_\bullet}{p}{\Fib(*;F)_\bullet}.\]
\end{enumerate}
\end{lem}

\begin{notat}
Here and in the following, when referring to a a space of lifts, we will always implicitly assume that the vertical map has been converted into a fibration (Kan
fibration for simplicial sets or Hurewicz fibration for topological spaces), using the standard path-space construction (between Kan complexes, in the case of simplicial sets). 
\end{notat}

\begin{proof}
(i) We only treat the case of bundles; the other case is completely analogous.
We only need to give a natural bijection 
\[D\colon \Bun_n(\vert X_\bullet\vert\times Y;F)_0 \to
\map_0(X_\bullet,\Bun_n(Y;F)_\bullet)\]
on the level of 0-simplices.

Let $q\colon E\to \vert X_\bullet\vert\times Y$ be a $0$-simplex in the left hand
side. We need to associate to it a simplicial map $X_\bullet\to\Bun_n(Y;F)_\bullet$.
Therefore let $\sigma$ be an $l$-simplex of $X_\bullet$, represented by a
simplicial map $\sigma\colon \Delta^l_\bullet\to X_\bullet$. The pull-back
$(\vert\sigma\vert\times\id_Y)^* q$ is then a bundle over $\Delta^l\times Y$
which defines an $l$-simplex in $\Bun_n(Y;F)_\bullet$.

Here is a description of the inverse $D'$ of $D$. Let $\phi_\bullet\colon
X_\bullet\to\Bun_n(Y;F)_\bullet$ be a simplicial map. For a nondegenerate
$k$-simplex $\tau$ of $\Bun_n(Y;F)_\bullet$, denote by $E(\tau)$ the total space
of the corresponding bundle over $B(\tau)=Y\times\Delta^k$. 

Now define $E\to \vert X_\bullet\vert\times Y$ to be the bundle
\[\bigcup_\sigma E(\phi(\sigma)) \to \bigcup_\sigma B(\phi(\sigma)),\]
with $\sigma$ ranging over the nondegenerate simplices of $X_\bullet$. If
$X_\bullet$ is finite, then this is a bundle by the Amalgamation property.
Otherwise use the fact that there is an open cover of $\vert X_\bullet\vert$ such
that each element of the cover is contained in a finite simplicial subset.

The total space is canonically a subset of $\vert X_\bullet\vert\times Y\times
\calU$, so the bundle is really a zero simplex in $\Bun_n(\vert X_\bullet\vert\times
Y;F)_\bullet$ and the map $D'$ is a strict inverse of $D$. 

(ii) follows from (i). 
\end{proof}

Let
\[\Fib(B;F):=\vert\Fib(B;F)_\bullet\vert;\quad \Bun_n(B;F):=\vert\Bun_n(B;F)_\bullet\vert\]

\begin{cor}\label{strspaces:structure_space_as_space_of_lifts}
If $B$ is a locally finite ordered simplicial complex, then there is a natural
weak homotopy equivalence
\[\calS_n(p)\to\lift{\Bun_n(*;F)}{B}{p}{\Fib(*;F)}\simeq \lift{\coprod_{[M]} B\TOP(M)}{B}{p}{BG(F)}.\]
\end{cor}

\begin{rem}
There is still a weak homotopy equivalence, well-defined up to homotopy, if $B$ is
homotopy equivalent to a locally finite ordered simplicial complex, as both
domain and target are homotopy invariant. 
\end{rem}

\begin{defn}\label{strspaces:universal_bundles}	
\begin{enumerate}
\item  The \emph{universal ``bundle''} over $\tilde\calB^n:=\Bun_n(B;F)$ is the map
\[\tilde\calP^n\colon\tilde \calE^n\to\tilde \calB^n\]
which corresponds to the identity map on $\Bun_n(B;F)_\bullet$ under the construction in the proof of Lemma \ref{strspaces:disassembly_map}. It is a bundle over every locally finite subcomplex of $\tilde \calB^n$.
\item The \emph{universal ``bundle''} over $\calB:=\Fib(B;F)$ is the map
\[\calP\colon \calE\to \calB\]
which corresponds to the identity map on $\Fib(B;F)_\bullet$ under the construction in the proof of Lemma \ref{strspaces:disassembly_map}. It is a fibration over every locally finite subcomplex of $\calB$.
\end{enumerate}
\end{defn}


\section{Review of the parametrized \texorpdfstring{$A$}{A}-theory characteristic}\label{section:DWW}

Let $p\colon E\to B$ be a fibration with homotopy finitely dominated fibers. Given such a fibration, and assuming that $B=\vert B_\bullet\vert$ is the geometric realization of a simplicial set, Dwyer-Weiss-Williams \cite{Dwyer-Weiss-Williams(2003)} associate the parametrized $A$-theory characteristic
\[\chi(p)\in\holim_{\sigma\in\simp B_\bullet}A(E_\sigma).\]
Here $\simp B_\bullet$ is the category of simplices of the simplicial set $B_\bullet$ and $E_\sigma=\vert\sigma\vert^*E$, the pull-back of $E$ to a fibration over the simplex $\vert\sigma\vert$. See the first subsection below for the models of $A$-theory that will be used. The definition of $\chi(p)$ will be reviewed after that.

The homotopy limit can be understood, up to homotopy, as a space of sections of a fibration over $B$ which is obtained from $p$ by applying the $A$-theory functor fiberwise:
\[\holim_{\sigma\in\simp B_\bullet}A(E_\sigma)\simeq\sections{A_B(E)}{B}.\]
See the third subsection below for this identification. 

If $p$ has the structure of a fiber bundle, with fibers compact topological manifolds, then $\chi(p)$ may be lifted over the fiberwise assembly map to an ``excisive characteristic'' $\chi^\%(p)$. This is briefly described in the forth subsection below. We conclude by some remarks on naturality.

\subsection{The models for $A(X)$ and $A^\%(X)$}
$\Rfd(X)$ denotes the Waldhausen category of homotopy finitely dominated spaces over $X$, with $w\Rfd(X)$ the subcategory of weak equivalences. By definition, $A(X):=K(\Rfd(X))$ is the algebraic $K$-theory space of that category (which is actually an infinite loop space). There is a natural transformation
\[\vert w\Rfd(X)\vert\to A(X)\]
reminiscent of the group completion, which will play an important role.

By $A^\%(X)$ we will always mean the homology theory of $X$ with coefficients in (the spectrum) $A(*)$, and by $\alpha\colon A^\%(X)\to A(X)$ the assembly map. The model for $A^\%(X)$ that will be used is the one from \cite[page 57]{Dwyer-Weiss-Williams(2003)} (see also \cite[Definition 2.3]{Badzioch-Dorabiala(2007)}), which is
\[A^\%(X)=\holim\bigl(A(X) \rightarrow K(\calR^{ld}(\mathbb{J}X)) \leftarrow K(\mathcal{V}(X))\bigr)\]
which will not be explained here. In any case the assembly transformation $A^\%(X)\to A(X)$ is given by the obvious projection. As in \cite{Badzioch-Dorabiala(2007)}, the Waldhausen category $\calR^\%(X)$ will denote the pull-back of categories
\[\Rfd(X) \to \calR^{ld}(\mathbb J X) \leftarrow \mathcal V(X).\]
It has the feature that it comes with a natural map
\[\vert w\calR^\%(X)\vert \to A^\%(X)\]
under which the projection to $w\Rfd(X)$ corresponds to the assembly map.

\subsection{Background on characteristics}\label{subsection:background_on_characteristics}
Denote by $\cat$ the category of small categories. Given a small category $\calC$ and a functor $F\colon \calC\to\cat$, a \emph{characteristic} for $F$ is a natural transformation
\[\chi\colon \calC/? \to F\]
where $\calC/c$ is the over category: An object is a morphism $d\to c$ in $\calC$ and a morphism is a commutative triangle. 

Unraveling the definitions, it is not hard to see that $\chi$ is given by the following data:
\begin{enumerate}
\item For each object $c$ of $\calC$, a ``characteristic object'' $c^!\in F(c)$, which corresponds to the image of the identity morphism on $c$ under the functor $\chi(c)$, and
\item For each morphism $\varphi\colon c\to d$ in $\calC$, a morphism $\varphi^!\colon \varphi_*(c^!)\to d^!$, satisfying the cocycle condition $(\psi\varphi)^!=\psi^!\circ \psi_*(\varphi^!)$. 
\end{enumerate}

In our context there are basically two ways how such characteristics occur:
\begin{enumerate}
\renewcommand{\labelenumi}{(\Alph{enumi})}
\item Let $f$ be a space-valued functor on $\calC$ such that all objects are mapped to homotopy finitely dominated spaces, and all morphisms are mapped to homotopy equivalences. From $f$ we obtain a characteristic for the functor $w\Rfd\circ f$ as follows: The  characteristic object is
\[c^!=f(c)\times S^0\in\Rfd(f(c)),\]
considered as a retractive space over $f(c)$. If $\varphi\colon c\to d$ is a morphism of $\calC$, then there is an obvious weak equivalence of retractive spaces
\[\varphi^!\colon \varphi_* c^! = f(c) \amalg f(d) \to d^! = f(d) \amalg f(d)\]
which is easily seen to satisfy the cocycle condition.

If $p\colon E\to B$ is as above, then this procedure can be applied to $\calC=\simp B_\bullet$, the category of simplices of the simplicial set $B_\bullet$. By definition, an object of $\simp B_\bullet$ is a simplex of $B$, i.e.~a simplicial map $\sigma\colon\Delta^n_\bullet\to B_\bullet$ for some $n$. A morphism from $\sigma$ to $\sigma'$ is a map $f\colon \Delta^n_\bullet\to\Delta^{n'}_\bullet$ such that $\sigma'\circ f=\sigma$. The category $\simp B_\bullet$ is the domain of a space-valued functor $f$ by
\[f(\sigma):=E_\sigma:=\vert\sigma\vert^*E.\]
The above procedure applies to give a characteristic for $F(\sigma)=w\Rfd(E_\sigma)$.

\item If $f$ as in (A) maps all objects even to compact ENRs, and maps all morphisms to cell-like maps between these, then we obtain a characteristic for the functor $w\calR^\%\circ f$. After applying the transformation $w\calR^\%\to w\calR$, one gets back the characteristic from (A).
\end{enumerate}

Now let $G\colon \calC\to\CGHaus$ be a space-valued functor. Then a characteristic for $G$ is defined to be a natural transformation
\[\chi\colon \vert\calC/?\vert\to G.\]
Obviously the geometric realization of a characteristic on a functor $F\colon\calC\to\cat$ as above defines a characteristic on $\vert F\vert$. In our construction (A) above, recall that there is a natural map
\[\vert w\Rfd(E_\sigma)\vert\to A(E_\sigma)\]
reminiscent of the group completion. Hence, by composition of natural transformations, we also obtain a characteristic for the functor
\[\sigma\mapsto A(E_\sigma).\]

The space of all characteristics for $G$, endowed with its canonical topology, is by definition the homotopy limit $\holim G$. In summary, we may define:

\begin{defn}
The \emph{parametrized $A$-theory characteristic} of $p$ is the element
\[\chi(p)\in \holim_{\sigma\in\simp B_\bullet} A(E_\sigma)\]
which corresponds to the characteristic obtained by geometric realization and ``group com\-pletion'' of construction (A) above.
\end{defn}

\subsection{Background on homotopy limits and section spaces}
For a functor $F$ from a small category $\calC$ to spaces,  a point $\chi\in\holim F$ is a characteristic for $F$, i.e.~a natural transformation $\vert \calC/?\vert\to F$. Such a natural transformation induces a lift
\[\xymatrix{
&& {\hocolim F} \ar[d] \\
{\hocolim \vert\calC/?\vert} \ar[rr]^{\pi_*} \ar@{.>}[rru]^{\chi_*} && {\hocolim *}
= \vert\calC\vert
}\]
of the canonical projection $\pi_*$.

As each of the spaces $\vert\calC/c\vert$ is contractible, the map $\pi_*$ is a homotopy equivalence; it induces a homotopy equivalence
\[
\pi^*\colon \sections{\hocolim F}{\vert \calC\vert}\xrightarrow{\simeq}
\lift{\hocolim F}{\hocolim\vert\calC/?\vert}{}{\vert\calC\vert}.
\]
(Recall that in our notation, all vertical maps have been converted into fibrations.) 

If $F$ sends all morphisms to homotopy equivalences, the projection map from $\hocolim F$ to $\vert\calC\vert$ is a quasifibration. In other words, if $c$ is an object of $\calC$, then the inclusion of $F(c)$ into the homotopy fiber over $c\in\vert\calC\vert$ of the projection map is a weak homotopy equivalence. One can show that in this case the above construction produces a zigzag of weak homotopy equivalences
\begin{equation}\label{whtors:weak_h_eq_from_holim_to_section_space}
\holim F\simeq\sections{\hocolim F}{\vert\calC\vert}.
\end{equation}

In the case of the parametrized $A$-theory characteristic, $\calC=\simp B_\bullet$, so $\vert \calC\vert\simeq B$ by Kan's last vertex map. The functor $F$ sends a simplex $\sigma\in\simp B_\bullet$ to the space $E_\sigma:=\vert\sigma\vert^*E$, the total space of the pull-back of $E$ over $\sigma$. In this situation the weak homotopy equivalence from \eqref{whtors:weak_h_eq_from_holim_to_section_space} takes the form
\[\holim_{\sigma\in\simp B_\bullet} A(E_\sigma) \simeq \sections{A_B(E)}{B}\] 
where, by definition, the map $A_B(E)\to B$ is the fibration associated with the composite $\hocolim F\to\vert\calC\vert\to B$, with fiber over $b\in B$ homotopy equivalent to $A(p^{-1}(b))$.

\subsection{The excisive characteristic}

If the fibration $p$ happens to be a bundle with fibers compact (not necessarily closed) topological manifolds, Dwyer-Weiss-Williams also define an excisive characteristic. Informally, it can be understood as a refinement of $\chi(p)$ in the sense that it defines, up to homotopy, an element in 
\[\chi^\%(p)\in\lift{A^\%_B(E)}{B}{\chi(p)}{A_B(E)}\]
where the vertical map is the fiberwise assembly map $\alpha$. Thus it defines a section, denoted by $\chi_e(p)$, of the fibration over $B$ obtained from $p$ by applying the functor $A^\%$ fiberwise, together with a path from $\alpha\chi_e(p)$ to $\chi(p)$. In the homotopy limit language this corresponds to an element in the homotopy fiber over $\chi(p)$ of the map
\begin{equation}\label{whtors:fiberwise_assembly_map}
\alpha\colon \holim_{\sigma\in\simp B_\bullet} A^\%(E_\sigma) \to \holim_{\sigma\in\simp B_\bullet} A(E_\sigma)
\end{equation}
induced by the assembly transformation.

Formally, the excisive characteristic is an element in yet another homotopy limit space we are going to define now. Let $tB_\bullet$ be the simplicial set where an $n$-simplex is an $n$-simplex $\sigma$ of $B_\bullet$, together with an equivalence relation $\theta$ on $E_\sigma:=\vert\sigma\vert^*E$, with quotient space $E_\sigma^\theta$, such that the projections induce a homeomorphism $E_\sigma\to \Delta^n\times E_\sigma^\theta$ over $\vert\sigma\vert\cong\Delta^n$. Then the functor from $\simp tB_\bullet$ to spaces, sending $(\sigma,\theta)$ to $E_\sigma^\theta$, maps all objects to compact ENRs and all morphisms to homeomorphisms (which are cell-like maps). Applying the construction (B) from above,  we obtain a characteristic for the functor
\[\simp tB_\bullet\to\cat, \quad (\sigma,\theta) \mapsto w\calR^\%(E_\sigma^\theta).\]

\begin{defn}
The \emph{excisive characteristic} of $p$ is the element
\[\chi_e(p)\in\holim_{(\sigma,\theta)} A^\%(E_\sigma^\theta)\]
obtained from the characteristic above by geometric realization and the natural transformation $\vert w\calR^\%(X)\vert\to A^\%(X)$.
\end{defn}

The images of $\chi(p)$ and $\chi_e(p)$ in $\holim_{(\sigma,\theta)}A(E^\theta_\sigma)$ are related by a canonical path. We obtain therefore an element  
\[\chi^\%(p)\in\holim\bigl(\holim_\sigma A(E_\sigma)\to \holim_{(\sigma,\theta)} A(E^\theta_\sigma) \leftarrow \holim_{(\sigma,\theta)} A^\%(E^\theta_\sigma)\bigr)
\]
which projects to $\chi(p)\in\holim_\sigma A(E_\sigma)$.

Now the projections 
\[\holim_\sigma A^\%(E_\sigma)\to\holim_{(\sigma,\theta)} A^\%(E_\sigma^\theta),\quad \holim_\sigma A(E_\sigma)\to\holim_{(\sigma,\theta)} A(E_\sigma^\theta)\]
are homotopy equivalences (this deep fact uses \cite{McDuff(1980)}, see \cite[Theorem 2.5]{Dwyer-Weiss-Williams(2003)}). So we may consider $\chi_e(p)$ as an element in the homotopy fiber of \eqref{whtors:fiberwise_assembly_map} over $\chi(p)$, well-defined up to homotopy.

\subsection{Some remarks on naturality}\label{subsection:some_remarks_on_naturality}

Suppose that, given a fibration $p\colon E\to B$ as above, the space $E$ is a subset of $B\times \calU$ and that the map $p$ corresponds to the projection from $B\times\calU$ to $B$. 

Let $f_\bullet\colon B'_\bullet\to B_\bullet$ be a simplicial map inducing $B'\to B$ on the geometric realization. Then $f_\bullet$ induces a functor $\simp B'_\bullet\to\simp B_\bullet$ so that we get a restriction map
\[f^*\colon \holim_{\sigma\in \simp B_\bullet} A(E_\sigma) \to\holim_{\sigma\in\simp B'_\bullet} A(E_\sigma)\]
If we use the convention that $E_\sigma$ is considered as a subset of $\vert\sigma\vert\times \calU$, then we have
\[f^*\chi(p)=\chi(p')\]
where $p'\colon f^*E\to B'$ is the pull-back fibration such that $f^*E$ is a subset of $B'\times \calU$. This follows immediately from the definitions, since $E'_\sigma = E_\sigma$.

If $p$ is a bundle of compact manifolds, then the same naturality statement holds for the excisive characteristic
\[\chi_e(p)\in\holim_{(\sigma,\theta)\in\simp tB_\bullet} A^\%(E_\sigma^\theta).\]

From now on we will tacitly assume that all bundles and fibrations $E\to B$ are in fact objects of the categories $\mathbf{Bun}_n(B;F)$ or $\mathbf{Fib}(B;F)$ as defined in section \ref{strspaces}, so that in particular the total spaces come as subsets of $B\times\calU$.

\begin{lem}\label{DWW:naturality_of_A_theory_characteristic}
Suppose that $\varphi\colon p\to p'$ is a fiber homotopy equivalence between fibrations for which the $A$-theory characteristic is defined. Then there is a canonical path
\[\varphi_*\chi(p)\rightsquigarrow \chi(p')\]
between the $A$-theory characteristics, where
\[\varphi_*\colon\holim_{\sigma} A(E_\sigma)\to\holim_{\sigma} A(E'_\sigma)\]
is the obvious map.
\end{lem}

\begin{proof}
The natural transformation $w\Rfd(E_\sigma)\to w\Rfd(E'_\sigma)$ may be viewed as a functor $G$ on the category $\simp B_\bullet\times [1]$ for which we may apply construction (A) again, since it sends all morphisms to homotopy equivalences. By geometric realization and ``group completion'' we obtain an element
\[\chi_\varphi\in\holim G=\holim\bigr(\holim_\sigma A(E_\sigma)\xrightarrow{\varphi_*} \holim_\sigma A(E'_\sigma)\bigl).\]
Such an element is given by a path from $\varphi_*\chi(p)$ to $\chi(p')$.
\end{proof}

See Lemma \ref{strspaces:lax_naturality_of_excisive_characteristic} for an analogous statement for the excisive characteristic.


\section{The parametrized Whitehead torsion}\label{whtors}

The goal of this section is to define a \emph{parametrized Whitehead torsion}
\[\tau\colon \calS_n(p)\to \sections{\Omega\Wh_B(E)}{B}\]
whenever $p\colon E\to B$ is a bundle of compact topological manifolds. Here the right-hand side is the space of sections of a fibration over $B$ which is obtained from $p$ by applying $\Omega\Wh$ fiberwise, as defined by \cite{Dwyer-Weiss-Williams(2003)} and reviewed in section \ref{section:DWW} above. The symbol $\Wh$ denotes the connective topological Whitehead functor as defined by Waldhausen. In the case where $B$ is a point, the parametrized Whitehead torsion reduces to
\[\tau\colon \calS_n(E)\to\Omega\Wh(E)\]
whose induced map on path components 
\[\pi_0\calS_n(E)\to\Wh(\pi_1(E))\]
sends a homotopy equivalence $f\colon M\to E$ to the classical Whitehead torsion $\tau(f)$ (Appendix \ref{section:unparametrized_case}).

The definition of the parametrized torsion can be outlined as follows: Denote by $\calP\colon\calE\to \calB$ the universal ``bundle'' for fibrations with fiber $F$ and by $\tilde\calP\colon \tilde\calE\to \tilde\calB$ the universal ``bundle'' for fiber bundles with fibers compact $n$-manifolds homotopy equivalent to $F$. (See definition \ref{strspaces:universal_bundles} for the notion of universal ``bundle''.) The parametrized $A$-theory characteristic and its excisive version yield a commutative diagram
\[\xymatrix{
{\tilde{\calB}} \ar[rr]^{\chi(\tilde\calP)} \ar[d] &&  A^\%_\calB(\calE) \ar[d]\\
{\calB} \ar[rr]^{\chi(\calP)} && A_\calB(\calE)
}\]
Hoehn \cite{Hoehn(2009)} used this diagram to produce a map
\[\calS_n(p)\xrightarrow{\simeq}\lift{\tilde\calB}{B}{p}{\calB}\to \lift{A^\%_\mathcal B(\mathcal E)}{B}{\chi(p)}{A_\mathcal B(\mathcal E)},\]
by precomposition. This map will be called \emph{parametrized excisive characteristic} by us. 

While the transition from the abstract homtopy limit spaces to section spaces makes this definition quite simple, it is inconvenient for our purposes. In fact the infinite loop space structure which is present in $K$-theory gets lost in the section space model. Therefore our first task will be to redescribe Hoehn's map completely within the homotopy limit picture:
\[\chi^\%\colon \calS_n(p)\to \hofib_{\chi(p)}\biggl(\holim_{\sigma\in\simp B_\bullet} A^\%(E_\sigma) \to \holim_{\sigma\in\simp B_\bullet} A(E_\sigma)\biggr).\] 

With this definition of the parametrized excisive characteristic, the parametrized Whitehead torsion is just one step away. In fact, if $p$ is a bundle of compact $n$-dimensional manifolds, we can subtract $\chi^\%(p)$ from the map $\chi^\%$ to get
\begin{multline*}\tau\colon\calS_n(p)\to \hofib_* \biggl(\holim_{\sigma\in\simp B_\bullet} A^\%(E_\sigma) \to \holim_{\sigma\in\simp B_\bullet} A(E_\sigma)\biggr) \\
\simeq \holim_{\sigma\in\simp B_\bullet} \Omega\Wh(E_\sigma) \simeq \sections{\Omega\Wh_\sigma(E_\sigma)}{B}.
\end{multline*}

Hence,
\[\bar R(p):=\holim_{\sigma\in\simp B_\bullet} \Omega\Wh(E_\sigma)\]
is the infinite loop space mentioned in the introduction. It is the space the parametrized Whitehead torsion is actually taking values in. When no confusion is possible, we will implicitly identify $\bar R(p)$ with the section space since this one may be considered to be the more intuitive one.

\subsection{Pull-back in families}

The goal of this subsection is to describe the technical framework which is necessary for the definition of the parametrized excisive characteristic $\chi^\%$. Recall that Hoehn's definition of the parametrized excisive characteristic used the Dwyer-Weiss-Williams construction in a universal situation. This will also be our guideline; however to avoid the passage to section spaces, a technically more complicated construction is necessary.

For notational convenience we remain in the abstract setting. The basic idea of this subsection is the following observation:

\begin{obs}\label{whtors:observation1}
Let $\calB_\bullet$ be Kan, $F\colon\simp \calB_\bullet\to\CGHaus$ be a functor that sends all morphisms to homotopy equivalences. Let $f,g\colon B_\bullet\to\calB_\bullet$ be homotopic via a homotopy $H$. Then there is a zig-zag of weak homotopy equivalences
\[\holim(F\circ f_*) \underset{H}{\simeq} \holim(F\circ g_*)\]
where $f_*$ and $g_*$ are the induced functors on the categories of simplices.

If $K$ is another homotopy from $g$ to $h$, then the zig-zag induced by a concatenated homotopy $H*K$ agrees with the composite zig-zag of weak homotopy equivalences.
\end{obs}

In fact, the zig-zag induced by $H\colon B_\bullet\times\Delta^1_\bullet\to\calB_\bullet$ is given by
\[\holim(F\circ f_*) \leftarrow \holim(F\circ H_*) \rightarrow \holim(F\circ g_*)\]
where the maps are given by restriction; they are weak homotopy equivalences by the section space interpretation.

Suppose now that we are given two maps of simplicial sets
\[f\colon \tilde\calB_\bullet\to\calB_\bullet, \quad p\colon B_\bullet\to \calB_\bullet\]
where $\calB_\bullet$ and $\tilde\calB_\bullet$ are Kan, and abbreviate
\[\calS := \lift{\tilde\calB_\bullet}{B_\bullet}{p}{\calB_\bullet}.\]
Let $F, F'\colon\simp \calB\to\CGHaus$ be two functors which send all morphisms to weak homotopy equivalences and $\alpha\colon F'\to F$ a natural transformation.

(In our main example, the spaces $\calB$ and $\tilde\calB$ are classifying spaces and $\calS\simeq\calS_n(p)$. The functors $F'$ and $F$ will send $\sigma$ to $A^\%(E_\sigma)$ and $A(E_\sigma)$ respectively, and $\alpha$ will be the assembly transformation.)

\begin{defn}
A \emph{characteristic pair} is an element $\chi\in\holim F$ together with a lift $\chi_e\in\holim(F'\circ f_*)$ through $\alpha$ of $f^*\chi$. 

A \emph{homotopy} from $(\chi,\chi_e)$ to another characteristic pair $(\chi',\chi'_e)$ is a path $\gamma$ from $\chi$ to $\chi'$ together with a lift of $f^*\gamma$ to a path from $\chi_e$ to $\chi'_e$.
\end{defn}

\begin{prop}\label{whtors:construction_of_excisive_characteristic}
\begin{enumerate}
\item Any characteristic pair $(\chi, \chi_e)$ determines, up to homotopy, a map
\[\kappa_{(\chi, \chi_e)}\colon \calS\to
\hofib_{p^*\chi} \bigl(\holim (F'\circ p_*)\xrightarrow{\alpha} \holim (F\circ p_*)\bigr).
\]
\item If $\gamma$ is a homotopy from $(\chi,\chi_e)$ to $(\chi',\chi'_e)$, then
\[\kappa_{(\chi', \chi'_e)}\simeq t_\gamma\circ \kappa_{(\chi, \chi_e)}\]
where $t_\gamma$ denotes fiber transport along $\gamma$.
\item Let $G,G'\colon\simp\calB\to\CGHaus$ be two more functors and 
\[\xymatrix{
F' \ar[rr]^\alpha \ar[d]^{\varphi} && F \ar[d]^\varphi\\
G' \ar[rr]^\beta && G 
}\]
a commutative square of natural transformations. Then $(\varphi_*\chi, \varphi_*\chi_e)$ is a characteristic pair for $(G,G')$ and
\[\kappa_{(\varphi_*\chi,\varphi_*\chi_e)} \simeq \varphi_*\kappa_{(\chi,\chi_e)}.\]
\item Suppose that $F, F'$ take values in loop spaces and that $\alpha$ is a loop map. If a characteristic pair $(\chi, \chi_e)$ decomposes as a loop sum $(\chi',\chi_e')+(\chi'',\chi''_e)$, then
\[\kappa_{(\chi, \chi_e)}=\kappa_{(\chi', \chi'_e)}+\kappa_{(\chi'', \chi''_e)}.\]
\end{enumerate}
\end{prop}

The remainder of this subsection will be devoted to the proof of this result. The core is the construction of a space $X$, equipped with a zigzag of weak homotopy equivalences
\begin{equation}\label{eq:complicated_mapping_space}
X\simeq \map\left(\calS, \holim(F\circ p_*)\right), 
\end{equation}
and a canonical map
\[\kappa\colon \holim (F\circ f_*)\to X.\]

Applying this construction to the functor $F'$ instead of $F$, it follows that any $\chi_e\in\pi_0\holim(F'\circ f_*)$ determines a homotopy class of maps
\[\calS\to\holim (F'\circ p_*),\]
which is half-way in the construction of $\kappa_{(\chi,\chi_e)}$.

To construct $X$, let $\calC:=\simp \calS\op$; an object in $\calC$ is thus of the form $(q,H)$ with 
\[
q\colon B_\bullet\times\Delta^n_\bullet\to\tilde\calB_\bullet,\quad
H\colon B_\bullet\times\Delta^n_\bullet\times\Delta^1_\bullet \to\calB_\bullet 
\]
where $q$ lifts $H_0$ and $H_1$ factors through $p$. Consider the functor $G\colon\calC\to\CGHaus$ given by
\[G(q,H)  = \holim(F\circ f_*\circ q_*).\]
By Observation \ref{whtors:observation1}, the homotopy  $H$ from $q\circ f$ to $p$ induces a (natural) weak homotopy equivalence $G(q,H)\simeq\holim(F\circ p_*)$ so that
\[\holim G \simeq \holim\bigl(\holim(F\circ p_*)\bigr)\simeq \map(\vert\calC\vert, \holim(F\circ p_*)).\]
Now $\vert\calC\vert\simeq\calS$; thus if we let $X:=\holim G$, we see that $X\simeq\map(\calS,\holim(F\circ p_*))$. 

For any object $(q,H)$ of $\calC$ there is a restriction map 
\[\holim (F\circ f_*)\to \holim(F\circ f_*\circ q_*)=G(q,H).\]
These maps are compatible with the morphisms of $\calC$, so we may define
\[\kappa\colon \holim (F\circ f_*)\to\lim G\to\holim G\]
where the latter map is the canonical inclusion. This finishes the construction of $X$ and $\kappa$.

\begin{proof}[of Proposition \ref{whtors:construction_of_excisive_characteristic}]
Under the zig-zag of maps between $G(q,H)$ and $\holim(F\circ p_*)$, the image of $f^*\chi$ corresponds to  $p^*\chi$. In other words there is a commutative diagram
\[\xymatrix{
{\holim(F'\circ f_*)} \ar[r]^\alpha \ar[d]^{\kappa} 
 & {\holim(F\circ f_*)} \ar[d]^\kappa
 & {*} \ar[l]_(.3){f^*\chi} \ar@{=}[d]\\
{\holim G'} \ar[r]^\alpha \ar@{-}[d]^{\wr}
 & {\holim G} \ar@{-}[d]^{\wr} 
 & {*} \ar[l]_(.3){\kappa(f^*\chi)} \ar@{=}[d]\\
{\map(\calS, \holim(F'\circ p_*))} \ar[r]^\alpha
 & {\map(\calS, \holim(F\circ p_*))} 
 & {*} \ar[l]_(.2){p^*\chi}
}\]
Here the lower horizontal arrow on the right denotes the constant map with value $p^*\chi$.

Thus there is a map
\begin{multline*}
\pi_0\hofib_{f^*\chi}\bigl(\holim(F'\circ f_*)\xrightarrow{\alpha}\holim(F\circ f_*)\bigr)\\
\xrightarrow{\kappa} 
\pi_0\hofib_{\kappa(f^*\chi)}\bigl(\holim G'\xrightarrow{\alpha}\holim G\bigr) \\
\cong
\bigl[\calS, \hofib_{p^*\chi} \bigl(\holim(F'\circ p_*)\xrightarrow{\alpha}\holim(F\circ p_*)\bigr)\bigr]. 
\end{multline*}
By definition, $\kappa_{(\chi,\chi_e)}$ is the image of $\chi_e$ under this map.

This proves part (i). To prove part (ii), replace the one-point spaces in the right-hand column of the diagram by intervals. Part (iii) follows from naturality and (iv) holds as all of these constructions are compatible with loop spaces.
\end{proof}

\subsection{The parametrized excisive characteristic and the parametrized torsion}

Recall from section \ref{strspaces} that $\mathcal B:=\vert\Fib(*;F)_\bullet\vert$ and $\tilde{\mathcal B}:=\vert\Bun_n(*;F)_\bullet\vert$ carry universal ``bundles'' $\mathcal P\colon \mathcal E\to \mathcal B$ and $\tilde{\mathcal P}\colon\tilde{\mathcal E}\to\tilde{\mathcal B}$. They have the property that the restriction of the ``bundles'' over every locally finite subcomplex are fibrations resp.~bundles. This is good enough to define parametrized characteristics, which only make use of the restrictions over simplices. Thus we get
\[\chi\in\holim_{\sigma\in\simp \calB_\bullet} A(\mathcal E_\sigma),\quad \chi_e\in\holim_{\sigma\in\simp\tilde\calB_\bullet} A^\%(\mathcal E_\sigma)\]
where $\chi_e$ lifts the restriction of $\chi$ over $\tilde\calB_\bullet$.

In other words, if we let
\[F(\sigma):=A(\mathcal E_\sigma),\quad F'(\sigma):=A^\%(\mathcal E_\sigma)\]
and $\alpha\colon F'\to F$ be the assembly transformation, then $(\chi,\chi_e)$ defines a characteristic pair.

\begin{defn}
The \emph{parametrized excisive characteristic} for the fibration $p$ is given by
\[\chi^\%:=\kappa_{(\chi, \chi_e)}\colon\calS_n(p)
\to
\hofib_{\chi(p)} \biggl(\holim_{\sigma\in\simp B_\bullet} A^\%(E_\sigma) \to \holim_{\sigma\in\simp B_\bullet} A(E_\sigma)\biggr).
\]
\end{defn}

So given two structures $x$ and $y$ on $p$, we may take the loop space difference
\begin{multline*}
\chi^\%(x)-\chi^\%(y)\in \hofib_* \biggl(\holim_{\sigma\in\simp B_\bullet} A^\%(E_\sigma) \to \holim_{\sigma\in\simp B_\bullet} A(E_\sigma)\biggr)\\
\simeq \holim_{\sigma\in\simp B_\bullet} \Omega\Wh(E_\sigma)\simeq \sections{\Omega\Wh_B(E)}{B}.
\end{multline*}

We will sometimes use the section space notation as it is the more intuitive one; strictly speaking however we will always use the homotopy limit space since this one has an infinite loop space structure.

\begin{defn}
If $p$ is a bundle of compact topological $n$-manifolds, the loop space difference 
\[\tau:=\chi^\% - \chi^\%(\id)\colon\calS_n(p)\to \holim_{\sigma\in\simp B_\bullet} \Omega\Wh(E_\sigma)\]
is the \emph{parametrized Whitehead torsion} for $p$.
\end{defn}

\subsection*{Naturality properties}

It follows directly from the definition of the excisive characteristic that it is compatible with pull-backs: If $f\colon B'\to B$ is a map, then we have
\[f^*\circ\chi^\% \simeq \chi^\%\circ f^*\]
with the obvious restriction operations. Since $f^*\id=\id$, the parametrized Whitehead torsion is compatible with pull-backs, too. 

Let us now consider naturality with respect to fiber homotopy equivalences: Let $p'\colon E'\to B$ and $p''\colon E''\to B$ be fibrations for which the excisive characteristic is defined, and let $\varphi\colon p'\to p''$ be a fiber homotopy equivalence. 
There is a canonical path $\varphi_*\chi(p')\rightsquigarrow \chi(p'')$; using the fiber transport $t_\gamma$ along this path we may define
\begin{multline*}
\varphi_\bigstar:=t_\gamma\circ \varphi_*\colon \hofib_{\chi(p')}\bigl(\holim_\sigma A^\%(E'_\sigma)\to\holim_\sigma A(E'_\sigma)\bigr) \\
\to \hofib_{\chi(p'')}\bigl(\holim_\sigma A^\%(E''_\sigma)\to\holim_\sigma A(E''_\sigma)\bigr).
\end{multline*}

\begin{lem}\label{whtors:covariant_functoriality_of_excisive_characteristic}
The parametrized excisive characteristic is natural in the sense that $\varphi_\bigstar\circ\chi^\%\simeq \chi^\%\circ\varphi_*$.
\end{lem}

\begin{cor}[(Composition rule)]
If $\varphi\colon p'\to p''$ is a fiber homotopy equivalence between bundles of compact $n$-manifolds, then
\[\tau\circ\varphi_* \simeq \varphi_*\circ \tau + \tau(\varphi).\]
\end{cor}

\begin{proof}[of Composition rule, assuming Lemma \ref{whtors:covariant_functoriality_of_excisive_characteristic}]
We have
\begin{align*}
\tau\circ \varphi_* & \simeq \chi^\%\circ \varphi_* - \chi^\%(\id_{p''}) \\
& \simeq(\chi^\%\circ \varphi_* - \chi^\%\circ\varphi_*(\id_{p'})) +
   (\chi^\%(\varphi) - \chi^\%(\id_{p''}))\\
& \simeq (\varphi_\bigstar\circ\chi^\% - \varphi_\bigstar\circ\chi^\%(\id_{p'})) + \tau(\varphi)\\
& \simeq (\varphi_*\circ\chi^\% - \varphi_*\circ\chi^\%(\id_{p'})) + \tau(\varphi)\\
& \simeq \varphi_*\circ \tau + \tau(\varphi). 
\end{align*}
\end{proof}

To prove Lemma \ref{whtors:covariant_functoriality_of_excisive_characteristic}, let $p\colon E\to B\times I$ be a fibration which restricts to $p'$ over $B\times 0$ and to $p''$ over $B\times 1$, such that fiber transport from $B\times 0$ to $B\times 1$ is given by the fiber homotopy equivalence $\varphi$. Denote by $i_0,i_1$ the inclusions of $B\times 0$ and $B\times 1$ into $B\times I$ respectively. Abbreviate $\bar R(p)$ as the target of the map $\chi^\%$ for the fibration $p$.

\begin{obs}
We have 
\[i_1^*\simeq \varphi_\bigstar\circ i_0^*\colon \bar R(p)\to \bar R(p'')\]
where the sign $\simeq$ means that the two maps agree after inverting weak homotopy equivalences.
\end{obs}

In fact, if the bundle $p=p'\times \id_I$ is trivial in the $I$-direction, then both $i_1^*$ and $i_0^*$ are left inverses of taking product with $I$. So $i_1^*\simeq i_2^*$. Since $\varphi_\bigstar=\id$ in this case, the claim is true. The general case follows since each fibration may be trivialized in the $I$-direction.

\begin{proof}[of Lemma \ref{whtors:covariant_functoriality_of_excisive_characteristic}]
With the same notation as above, we have
\[\chi^\%\circ \varphi_*\circ i_0^*\simeq \chi^\%\circ i_1^*\simeq i_1^*\circ \chi^\%\simeq \varphi_\bigstar \circ i_0^*\circ \chi^\%\simeq \varphi_\bigstar\circ \chi^\%\circ i_0^*.\]
Since $i_0^*$ is a homotopy equivalence, the result follows.
\end{proof}

As a final naturality property, we remark:

\begin{rem}[(Homeomorphism invariance)]
If $\varphi\colon E'\to E$ is a fiber homeomorphism of bundles, then $\tau(\varphi)=0$. In fact, the fiber homeomorphism provides $\varphi\in\calS_n(p)$ with a canonical nullhomotopy.
\end{rem}

\section{A product rule}\label{section:product_rule}

Given a fibration $p\colon E\to B$ and a compact topological $k$-manifold $X$, taking cartesian product with $X$ defines a map
\[(\times X)\colon \calS_n(p)\to\calS_{n+k}(p\times X)\]
where $p\times X\colon E\times X\to B$ is the composition of the projection map onto $E$ with $p$. We will show:

\begin{thm}\label{pr:general_product_rule}
If $X$ is connected, then the parametrized Whitehead torsion satisfies:
\[\tau\circ (\times X) \simeq \chi_e(X) \cdot i_*\circ \tau\colon \calS_n(p)\to\sections{\Omega\Wh_B(E\times X)}{B}\]
where $\chi_e(X)\in\ZZ$ is the Euler characteristic and $i_*$ is the map induced by the inclusion $E\cong E\times \{x\}\subset E\times X$ for any $x\in X$.
\end{thm}

Recall that, strictly speaking, the space of sections appearing in the right-hand side is a certain homotopy limit which carries the structure of an infinite loop space. This structure is used for the multiplication in the Theorem.

\begin{rem}
Recall that the total spaces of all our bundles and fibrations are subsets of $B\times\calU$. To make sense of the product map $(\times X)$, we must (and will) assume that $X$ is a subset of $\calU$, so that $E\times X\subset B\times \calU\times \calU\cong B\times \calU$, by our chosen bijection $\calU\times\calU\to\calU$.
\end{rem}

In this section, we will only consider the case where $X$ is contractible -- the general case will follow from the additivity of the parametrized torsion, see section \ref{section:additivity}. If $X$ is contractible, then $i_*$ is a homotopy equivalence, whose inverse is induced by the projection $\pi\colon E\times X\to E$. Hence in this situation Theorem \ref{pr:general_product_rule} reads:
\begin{equation}\label{pr:special_product_rule}
\pi_*\circ \tau\circ (\times X) \simeq \tau\colon\calS_n(p)\to\sections{\Omega\Wh_B(E)}{B}.
\end{equation}

In fact \eqref{pr:special_product_rule} is a consequence of a corresponding property of the parametrized excisive characteristic. Recall the  map $\pi_\bigstar$ which was defined before Lemma \ref{whtors:covariant_functoriality_of_excisive_characteristic}. It connects the target of the excisive characteristic for $p\times X$ with that for $p$ in the most obvious way.

\begin{thm}\label{pr:special_product_rule_for_excisive_characteristic}
We have
\[
\pi_\bigstar\circ\chi^\%\circ (\times X) \simeq \chi^\%\colon
\calS_n(p)
\to 
\hofib_{\chi(p)}\biggl(\holim_{\sigma\in\simp B_\bullet} A^\%(E_\sigma)\to \holim_{\sigma\in\simp B_\bullet} A(E_\sigma)\biggr).
\]
\end{thm}
 
Using this Theorem, we conclude
\begin{align*}
\pi_*\circ \tau\circ (\times X) 
& \simeq \pi_\bigstar\circ \chi^\%\circ(\times X) - \pi_\bigstar\circ\chi^\%(\id_{p\times X})\\
& \simeq \pi_\bigstar\circ \chi^\%\circ(\times X) - \pi_\bigstar\circ\chi^\%\circ(\times X)(\id_p)\\
& \simeq \chi^\%-\chi^\%(\id_p)\\
& \simeq \tau,
\end{align*}
hence we obtain \eqref{pr:special_product_rule}, so that Theorem \ref{pr:general_product_rule} follows in the case when $X$ is contractible.

The remainder of this section will be devoted to the proof of Theorem \ref{pr:special_product_rule_for_excisive_characteristic}. Recall that
\[\chi^\%=\kappa_{(\chi,\chi^n_e)}\]
where $(\chi,\chi^n_e)$ is the characteristic pair for the universal ``bundle'' $\calP$ over $\calB$ and the universal ``bundle'' of compact $n$-manifolds $\tilde\calP^n$ over $\tilde\calB^n$, and $\kappa_{(\chi,\chi^n_e)}$ is given as in Proposition \ref{whtors:construction_of_excisive_characteristic}.

Now consider the ``bundles'' $\calP\times X$ and $\tilde\calP^n\times X$; they also produce a characteristic pair $(\chi',\chi'_e)$ for the functors
\[G'(\sigma):=A^\%(E_\sigma\times X),\quad G(\sigma):=A(E_\sigma\times X)\]
on $\simp \calB$.

The main part of the proof will be to show that the characteristic pairs $(\chi,\chi^n_e)$ and $(\pi_*\chi', \pi_*\chi'_e)$ are homotopic. Assuming this for the moment, the proof of Theorem \ref{pr:special_product_rule_for_excisive_characteristic} is a formal consequence:

\begin{proof}[of Theorem \ref{pr:special_product_rule_for_excisive_characteristic}]
Since the product ``bundles'' $\calP\times X$ and $\tilde\calP^n\times X$ are classified by maps to $\calP$ and $\calP^{n+k}$ respectively, the characteristic pair $(\chi', \chi'_e)$ is obtained by restriction of the characteristic pair $(\chi,\chi^{n+k}_e)$. It follows that
\[\kappa_{(\chi',\chi'_e)} \simeq \kappa_{(\chi,\chi^{n+k}_e)}\circ(\times X).\]

Denote by $\gamma$ the path from $(\pi_*\chi', \pi_*\chi'_e)$ to $(\chi, \chi_e)$ whose existence we assume. Using the above identity together with parts (ii) and (iii) of Proposition \ref{whtors:construction_of_excisive_characteristic}, we have
\begin{align*}
\multbox
\chi^\% & \simeq \kappa_{(\chi,\chi^n_e)}\\
& \simeq t_\gamma \circ \kappa_{(\pi_*\chi',\pi_*\chi'_e)}\\
& \simeq t_\gamma\circ\pi_* \circ \kappa_{(\chi',\chi'_e)}\\
& \simeq \pi_\bigstar \circ \kappa_{(\chi,\chi^{n+k}_e)}\circ (\times X)\\
& \simeq \pi_\bigstar \circ\chi^\%\circ (\times X).
\emultbox
\end{align*}
\end{proof}

Thus it remains to show that the characteristic pairs $(\chi,\chi_e^n)$ and $(\pi_*\chi', \pi_*\chi'_e)$ are indeed homotopic. Recall that by Lemma \ref{DWW:naturality_of_A_theory_characteristic}, there is a canonical path $\pi_*\chi'\to\chi$. Moreover we have:

\begin{lem}\label{strspaces:lax_naturality_of_excisive_characteristic}
\begin{enumerate}
\item Let $X$ be a contractible compact manifold, $p\colon E\to B$ a bundle of compact topological manifolds, and consider the bundle $p\times X\colon E\times X\to B$. The canonical path 
\[\pi_*\chi(p\times X)\rightsquigarrow \chi(p)\]
between the parametrized $A$-theory characteristics, induced by the projection $\pi\colon E\times X\to E$, lifts along the fiberwise assembly map $\alpha$, up to homotopy relative endpoints, to a path 
\[\pi_*\chi_e(p\times X)\rightsquigarrow \chi_e(p)\]
between the Dwyer-Weiss-Williams excisive characteristics.
\item Let $Y$ be another contractible compact manifold. If $\pi'\colon E\times X\times Y\to E\times X$ is the projection, then the two possible ways of lifting the path
\[\pi'_*\pi_*\chi(p\times X\times Y)\rightsquigarrow  \chi(p)\]
are homotopic relative endpoints. 
\end{enumerate}
\end{lem}

\begin{rem}
Recall that, formally, the characteristic 
\[\chi_e(p)\in\holim_{(\sigma,\theta)\in\simp tB_\bullet} A^\%(E_\sigma^\theta)\]
is a characteristic for a functor defined on the category of simplices on a simplicial set $tB_\bullet$, where a simplex is a pair $(\sigma,\theta)$ of a simplex $\sigma$ of $B_\bullet$ and an equivalence relation $\theta$ on $E_\sigma$. Consequently, 
\[\chi_e(p\times X)\in\holim_{(\sigma,\theta')\in\simp t'B_\bullet} A^\%((E_\sigma \times X)^{\theta'})\]
where a simplex of $t'B_\bullet$ is a simplex $\sigma$ of $B_\bullet$ and an equivalence relation $\theta'$ on $E_\sigma \times X$. Notice that there is an inclusion of simplicial sets
\[i\colon tB_\bullet\to t'B_\bullet,\quad (\sigma,\theta) \mapsto (\sigma, \theta\times X)\]
where $\theta\times X$ is the product relation of $\theta$ with the identity relation on $X$.

The map $\pi_*$ used in Lemma \ref{strspaces:lax_naturality_of_excisive_characteristic} is defined to be the composite
\[\holim_{(\sigma,\theta')\in\simp t'B_\bullet} A^\%((E_\sigma\times X)^{\theta'})
   \xrightarrow{i^*}
\holim_{(\sigma,\theta)\in\simp tB_\bullet} A^\%(E_\sigma^\theta\times X)
\xrightarrow{\bar\pi_*}
\holim_{(\sigma,\theta)\in\simp tB_\bullet} A^\%(E_\sigma^\theta)
\]
where the second map is induced by the projection $E_\sigma^\theta\times X\to E_\sigma^\theta$.

The diagram
\[\xymatrix{
{\holim_{\sigma\in\simp B_\bullet}  A^\%(E_\sigma\times X)} \ar[rr]^{\pi_*} \ar[d]^\simeq
&& {\holim_{\sigma\in\simp B_\bullet}  A^\%(E_\sigma)} \ar[d]^\simeq
\\
{\holim_{(\sigma,\theta')\in\simp t'B_\bullet} A^\%((E_\sigma\times X)^{\theta'})} \ar[rr]^{\pi_*}
&& \holim_{(\sigma,\theta)\in\simp tB_\bullet} A^\%(E_\sigma^\theta)
}\]
is commutative, with vertical maps which are homotopy equivalences (for non-trivial reasons, as mentioned before). 
\end{rem}

\begin{proof}[of Lemma \ref{strspaces:lax_naturality_of_excisive_characteristic}]
We can view the natural transformation
\[E_\sigma^\theta\times X \to E_\sigma^\theta\]
as a functor $F$ on the category $\simp tB_\bullet\times [1]$ which sends all objects to compact ENRs. As $X$ is a contractible compact manifold, $F$ also sends all morphisms to cell-like maps. Following the procedure of section \ref{subsection:background_on_characteristics}, we obtain an element 
\[\bar\chi\in \holim(F) = \holim\biggl(\holim_{(\sigma,\theta)\in\simp tB_\bullet} A^\%(E_\sigma^\theta\times X)
\xrightarrow{\bar\pi_*}
\holim_{(\sigma,\theta)\in\simp tB_\bullet} A^\%(E_\sigma^\theta)\biggr).\]
Spelled out, $\bar\chi$ is the element $i^*\chi_e(p\times X)$ together with a path from $\pi_*\chi_e(p\times X)=\bar\pi_*i^*\chi_e(p\times X)$ to $\chi_e(p)$. This construction is obviously compatible with the construction of the path from $\pi_*\chi(p\times X)$ to $\chi(p)$ from section \ref{subsection:some_remarks_on_naturality}.

This shows part (i) of the Lemma. Part (ii) can be proven similarly by considering the excisive characteristic of a functor defined on $\simp tB_\bullet \times [2]$.
\end{proof}

\section{Additivity}\label{section:additivity}

One of the major properties of classical Whitehead torsion is additivity. Building on results of Badzioch-Dorabia\l a, we are going to show that the parametrized Whitehead torsion enjoys an analogous property. First we proceed to formulate the parametrized version. For a paracompact base space $B$, let us form a category where an object is an object of $\mathbf{Fib}(B;F)$ for some homotopy finitely dominated space $F$, and a morphism $p\to p'$ is a fiberwise map from $p$ to $p'$. With this language, let $\square$ denote the following commutative diagram in this category
\begin{equation}\label{strspaces:diagram_square}
\xymatrix{
p_0 \ar[r] \ar[d] \ar[rd]^{j_0} & p_1 \ar[d]^{j_1}\\
p_2 \ar[r]^{j_2} & p_3
}\end{equation}
where $p_i$ is an object of $\mathbf{Fib}(B;F_i)$ for $i=0,1,2,3$ and where we assume that all maps on the level of total spaces are cofibrations. We also assume that the total space $E(p_3)$ is the push-out of the total spaces $E(p_1)$ and $E(p_2)$ over $E(p_0)$.

By an $n$-dimensional structure on $\square$, we mean an extension of this diagram to a commutative cube 
\[\xymatrix@=1.5ex{
q_0 \ar[rr] \ar[dd] \ar[rd] && q_1  \ar[dd]\ar[rd] \\
& p_0 \ar[rr] \ar[dd] && p_1 \ar[dd]\\
q_2 \ar[rr] \ar[rd] && q_3 \ar[rd]\\
&p_2 \ar[rr] && p_3
}\]
such that $q_i$ are objects of the categories $\mathbf{Bun}_n(B;F_i)$ for $i=1,2,3$ respectively, $q_0$ is an object of $\mathbf{Bun}_{n-1}(B;F_0)$, all the maps $q_i\to p_i$ are fiber homotopy equivalences, and the $q$-square is a codimension 1 splitting of $q_3$. By this we mean that all bundles $q_i$ are locally flat subbundles of $q_3$, of codimension 0 for $i=1,2$ and of codimension 1 of $i=0$, and that the total space of $q_0$ is the intersection of the total spaces of $q_1$ and $q_2$. The $n$-dimensional structures on $\square$ form the zeroth term of a simplicial set $\calS_n(\square)_\bullet$, with $k$-simplices such diagrams parametrized over $\Delta^k$. Let $\calS_n(\square)$ be its geometric realization. It comes with forgetful maps $\beta_i\colon \calS_n(\square)\to\calS_n(p_i)$ for $i=1,2,3$, and $\beta_0\colon\calS_n(\square)\to\calS_{n-1}(p_0)$. 

\begin{thm}[(Additivity)]\label{strspaces:additivity}
The following diagram commutes up to homotopy:
\[\xymatrix{
{\calS}_n(\square) \ar[rr]^{\beta_3} \ar[d]^(0.4){\prod_{i=0}^2 \tau\circ \beta_i} && {\calS}_n(p_3) \ar[d]^(0.4)\tau\\
{\prod_{i=0}^2} \sections{\Omega\Wh_B(E_i)}{B} \ar[rr]^{j_{1*}+j_{2*}-j_{0*}} && {\sections{\Omega\Wh_B(E)}{B}}
}\]
\end{thm}

\begin{rem}
As pointed out in section \ref{whtors}, the section spaces appearing in the Additivity theorem are really certain homotopy limit spaces which carry the structure of infinite loop spaces. This structure is used to define the sum appearing in the lower line of the diagram. 
\end{rem}

The proof consists, first, of describing $\calS_n(\square)$ as a suitable space of lifts, by considering universal bundles and fibrations with splittings. This will follow from the machinery presented in section \ref{strspaces} if we can establish certain formal properties of the ``bundle theories'' involved. Then we will make use of additivity results for the Dwyer-Weiss-Williams homotopy invariant and excisive characteristics, obtained by \cite{Dorabiala(2002)} and \cite{Badzioch-Dorabiala(2007)}, respectively. 

The fibers of each of the fibrations in $\square$ fit into a diagram
\[\xymatrix{
F_0 \ar[r] \ar[d] & F_1 \ar[d]\\
F_2 \ar[r] & F_3
}\]
which we are going to abbreviate $(F_i)$. Now consider, for a space $B$, the category $\mathbf{Fib}(B;(F_i))$ of all diagrams \eqref{strspaces:diagram_square} satisfying the conditions stated there, such that there are compatible homotopy equivalences between the fibers of $p_i$ and $F_i$. Morphisms are commuting diagrams of fiber homotopy equivalences. Pulling back with $f\colon B'\to B$ defines an object in $\mathbf{Fib}(B';(F_i))$. In fact, the cofibration condition still holds on the induced fibrations by Corollary \ref{fibr:cofibration_preserved_under_pullback}.

Following the construction of section \ref{strspaces}, we obtain a functor \[\mathbf{Fib}(B;(F_i))\colon\cpCW\op\to\cat\]
and therefore a simplicial category. Finally we obtain a simplicial set 
\[\Fib(B;(F_i))_\bullet := N_0\mathbf{Fib}(B;(F_i))_\bullet.\]
Similarly consider the category $\mathbf{Bun}_n(X;(F_i))$ whose objects are those objects of $\mathbf{Fib}(B;(F_i))$ such that $p_0$ is an object of $\mathbf{Bun}_{n-1}(B;F_0)$, $p_i$ are objects in $\mathbf{Bun}_n(B;F_i)$ for $i=1,2,3$, and the square is a codimension 1 splitting of $p$. A morphism in this category is to be a commuting diagram of fiberwise homeomorphisms. This category gives rise to a simplicial set $\Bun_n(B;(F_i))_\bullet$. 

\begin{lem}
The functors $\mathbf{Bun}_n(B;(F_i))$ and $\mathbf{Fib}(B;(F_i))$ satisfy the Amalgamation property, the Straightening property, and the Fill-in property if $B$ is metrizable ULC.
\end{lem}

\begin{proof}
The case of bundles being similar to the case in section \ref{strspaces}, we focus on fibrations. For the Amalgamation property, we have to check that the cofibration condition still holds after glueing. This is verified in Proposition \ref{fibr:glueing_and_cofibrations}. For the Straightening property, first use homotopy lifting to straighten each of the fibrations, and then use cofibration arguments to make the diagrams strictly commutative. Finally, the Fill-in property is the content of Proposition \ref{fibr:fill_in_with_splitting}.
\end{proof}

Following the arguments of section \ref{strspaces} again, one sees that whenever $B$ is homotopy equivalent to a locally finite ordered simplicial complex, there is a weak homotopy equivalence
\[\calS_n(\square) \xrightarrow{\simeq} \lift{\Bun_n(*;(F_i))}{B}{\square}{\Fib(*;(F_i))}.\]
Denote by $\calQ$ and $\tilde\calQ$ the universal ``bundles'' over $\Fib(*;(F_i))$ and $\Bun(*;(F_i))$. They are of the form
\pagebreak[4]
\[\xymatrix{
{\calQ_0} \ar[d] \ar[r] \ar[rd]^{j_0} & {\calQ_1} \ar[d]^{j_1}
 && {\tilde\calQ_0} \ar[d] \ar[r] \ar[rd]^{j_0} & {\tilde \calQ_1} \ar[d]^{j_1}\\
{\calQ_2} \ar[r]^{j_2} & {\calQ_3}
 && {\tilde \calQ_2} \ar[r]^{j_2} & {\tilde \calQ_3}
}\]

\begin{thm}[({\cite{Dorabiala(2002),Badzioch-Dorabiala(2007)}})]
\label{whtors:additivity_of_DWW}
\begin{enumerate}
\item On the level of Dwyer-Weiss-Williams $A$-theory characteristics, there is a canonical path 
\[j_{1*}\chi(\calQ_1)+j_{2*}\chi(\calQ_2)-j_{0*}\chi(\calQ_0)
\rightsquigarrow\chi(\calQ_3).\]
\item The path from (i) lifts, up to homotopy relative endpoints, to a path  \[j_{1*}\chi_e(\tilde \calQ_1)+j_{2*}\chi_e(\tilde \calQ_2)-j_{0*}\chi_e(\tilde \calQ_0)\rightsquigarrow \chi_e(\tilde \calQ_3)\]
between the Dwyer-Weiss-Williams excisive characteristics.
\end{enumerate}
\end{thm}

Hence the restriction $(\chi',\chi'_e)$ of the excisive pair $(\chi,\chi_e)$ onto $\Bun_n(*,(F_i))$ and $\Fib(*;(F_i))$ decomposes as
\[(\chi',\chi'_e) = (\chi_1,(\chi_e)_1) + (\chi_2,(\chi_e)_2) - (\chi_0,(\chi_e)_0).\] Now part (iv) of Lemma \ref{whtors:construction_of_excisive_characteristic} shows that the parametrized excisive characteristic $\chi^\%$ behaves additively. The additivity for the parametrized torsion is again a formal consequence. We omit the details, which are very similar to the the proof of the product rule.

\begin{rem}
In \cite{Badzioch-Dorabiala(2007)}, the statement is only formulated for closed manifolds. However the proof actually shows the more general case of compact manifolds.
\end{rem}

As an application, we may now prove the general version of the Product rule using the special case proved in section \ref{section:product_rule} and the Additivity theorem:

\begin{proof}[of Theorem \ref{pr:general_product_rule}]
Suppose that $X$ comes provided with a codimension 1 splitting $X=Y \cup_A Z$, and suppose that the claim holds for $Y$, $A$, and $Z$. Then the claim also holds for $X$. In fact, denote by $p\times X\colon E\times X\to B$ the composite fibration and by $\square$ the following square:
\[\xymatrix{
p\times A \ar[r] \ar[d] \ar[rd]^{j_0} & p\times Y \ar[d]^{j_1}\\
p\times Z \ar[r]^{j_2} & p\times X
}\]
with the obvious inclusions. Then taking product with $A$, $Y$, $Z$, and $X$ induces a map
\[P\colon\calS_n(p)\to \calS_{n+k}(\square).\]
Using the Additivity theorem, we conclude:
\begin{align*}
\tau\circ (\times X)& =\tau\circ \beta_3\circ P \\
& \simeq j_{1*}\circ \tau\circ \beta_1\circ P + j_{2*}\circ\tau\circ\beta_2\circ P - j_{0*}\circ\tau\circ \beta_0\circ P\\
& = j_{1*}\circ \tau\circ (\times Y) + j_{2*}\circ\tau\circ(\times Z) - j_{0*}\circ\tau\circ (\times A)\\
& \simeq \chi_e(Y) \cdot i_*\circ\tau + \chi_e(Z) \cdot i_*\circ \tau - \chi_e(A)\cdot i_*\circ \tau\\
& = \chi_e(X) \cdot i_*\circ\tau,
\end{align*}
since the Euler characteristic behaves additively.

Applying this reasoning to the codimension 1 splitting $S^n= D^n \cup_{S^{n-1}} D^n$, one sees inductively that the claim holds for $X=S^n$ for all $n$. Using the special version of the product rule, Theorem \ref{pr:special_product_rule_for_excisive_characteristic}, it follows that the claim also holds for all spaces $X=D^l\times S^n$.

Let now $X$ be arbitrary. We may assume that $\dim X\geq 6$; for otherwise we may replace $X$ by $X\times D^6$. Choose a handlebody decomposition of $X$, which exists by \cite[Essay III, Thm. 2.1]{Kirby-Siebenmann(1977)}. The proof is by induction on the number of handles: If there is only one handle, then $X\cong D^n$, so we are in the special case that has already been proven. In the inductive step, we have a codimension one splitting $X=X'\cup_{D^l\times S^{k-1}} D^l\times D^k$ where the claim holds for all the three spaces by the inductive assumption. The above argument shows that the claim holds for $X$ as well.
\end{proof}

\section{The torsion on the stable structure space}\label{section:sttors}

So far we defined for each space $\calS_n(p)$, $n\in\NN$, a parametrized excisive characteristic $\chi^\%$ and a parametrized torsion map $\tau$, well-defined up to homotopy. Now let $I:=[0,1]$ be the unit interval and consider the stabilization map
\[(\times I)\colon \calS_n(p)\to \calS_{n+1}(p\times I).\]

It follows from section \ref{section:product_rule} that 
\begin{equation}\label{sttors:noncanonical_homotopies}
\pi_\bigstar\circ \chi^\%\circ (\times I)\simeq \chi^\%, \quad \pi_*\circ \tau\circ (\times I)\simeq \tau
\end{equation}
for the projection $\pi\colon p\times I\to p$. The goal of this section is to use these homotopies to define a stabilized version of $\chi^\%$ on
\[\calS_\infty(p):=\hocolim_n \calS_n(p\times I^n)\]
whose restriction onto any of subspaces $\calS_n(p\times I^n)\subset\calS_\infty(p)$ agrees with the unstable $\chi^\%$ defined so far. It follows that the parametrized Whitehead torsion extends to a map
\[\tau\colon\calS_\infty(p)\to\sections{\Omega\Wh_B(E)}{B}\]
as usual by taking the loop space difference with $\chi^\%(p)$.

Notice that the homotopies in \eqref{sttors:noncanonical_homotopies} are not canonical in any way, since for each $n$ the parametrized excisive characteristic is well-defined only up to homotopy. So we need to analyze the construction of $\chi^\%$ more carefully.

Recall that $\chi^\%$ was defined using characteristic pairs. In fact for each $n$ there is a characteristic pair $(\chi, \chi^n_e)$ where $\chi$ is the parametrized $A$-theory characteristic for the universal ``bundle'' $\calE\to\calB=\Fib(*;F)$ and $\chi^n_e$ is the Dwyer-Weiss-Williams excisive characteristic for the universal ``bundle'' $\tilde\calE^n\to\tilde\calB^n=\Bun_n(*;F)$. 

Denote by $\sigma\colon \calB\to\calB$ and $\tilde\sigma\colon \tilde\calB^n\to\tilde\calB^{n+1}$ the stabilization map $(\times I)$. By Lemma \ref{strspaces:lax_naturality_of_excisive_characteristic}, there is a homotopy of characteristic pairs
\[(\pi_*\sigma^*\chi^{n+1}, \pi_*\tilde\sigma^*\chi^{n+1}_e)\rightsquigarrow (\chi^n,\chi^n_e)\]
where $\pi$ denotes either of the projection maps
\[\calE\times I\to\calE,\quad \tilde\calE^n\times I\to\tilde\calE^n.\]
Letting $n$ vary we obtain a whole family of characteristic pairs which are connected by homotopies as above.

The following result (and its proof) is the canonical extension of Proposition \ref{whtors:construction_of_excisive_characteristic}, which dealt with the situation for a single characteristic pair. Some indications of proof will be given below.

\begin{prop}\label{whtors:construction_of_stable_excisive_characteristic}
Such a family of characteristic pairs determines, up to homotopy, a map
\[\calS_\infty(p)\to 
\hofib_{p^*\chi} \biggl(\holim_{\sigma\in\simp B_\bullet} A^\%(E_\sigma) \to \holim_{\sigma\in\simp B_\bullet} A(E_\sigma)\biggr)
\]
\end{prop}

\begin{defn}
\begin{enumerate}
\item The map in Proposition \ref{whtors:construction_of_stable_excisive_characteristic} is called the \emph{stable excisive characteristic.}
\item If $p$ is a bundle of compact topological $n$-manifolds, then the \emph{stable parametrized torsion} is the loop difference
\[\tau:=\chi^\%-\chi^\%(x)\colon \calS_\infty(p)\to\ \holim_{\sigma\in\simp B_\bullet} \Omega\Wh(E_\sigma)\simeq\sections{\Omega\Wh_B(E)}{B}\]
for the structure $x$ given by the inclusion $p\to p\times I^n$.
\end{enumerate}
\end{defn}

Working with families of characteristic pairs instead of characteristic pairs, one can see: 

\begin{thm}
The Composition rule, Additivity property, and the Product rule also hold for the stable parametrized torsion.
\end{thm}

\begin{proof}[of Proposition \ref{whtors:construction_of_stable_excisive_characteristic}]
We will only give the construction of the map
\[\calS_\infty(p)\to \holim_{\sigma\in\simp B_\bullet} A^\%(E_\sigma),\]
which is the first half of the construction, leaving the second half to the reader.

The different elements $\chi^n_e$, along with their homotopies, assemble to define an element in the homotopy limit of
\begin{equation}\label{whtors:complicated_homotopy_limit_space}
\holim_{\sigma\in\simp \tilde\calB^1_\bullet} A^\%(\tilde\calE^1_\sigma) \leftarrow \holim_{\sigma\in\simp \tilde\calB^2_\bullet} A^\%(\tilde\calE^2_\sigma) \leftarrow \dots
\end{equation}
where the connecting maps are induced by restriction along $\sigma$ followed by the map induced by $\pi$.

Recall that the main ingredient in the proof of Proposition \ref{whtors:construction_of_excisive_characteristic} was the construction of a space 
\[X^n\simeq \map(\calS_n(p), \holim_{\sigma\in\simp B_\bullet} A^\%(E_\sigma))\]
such that $\chi^n_e(p)$ determines a point in $X^n$. This construction is natural in $n$ so that we obtain a homotopy equivalence
\[
\holim_n X^n\simeq \holim_n \map(\calS_n(p), \holim_{\sigma\in\simp B_\bullet} A^\%(E_\sigma))
= \map(\calS_\infty(p), \holim_{\sigma\in\simp B_\bullet} A^\%(E_\sigma)).
\]
Moreover the family $(\chi^n_e)$, considered as an element in the homotopy limit of \eqref{whtors:complicated_homotopy_limit_space}, determines an element in $\holim_n X^n$ and therefore a map
\[\calS_\infty(p)\to \holim_{\sigma\in\simp B_\bullet} A^\%(E_\sigma)\]
up to homotopy.
\end{proof}

\subsection*{Hoehn's result}

Let $p\colon E\to B$ be a fibration. Given a 0-simplex in $\calS_n(p)$, i.e.~a commutative diagram
\[\xymatrix{
E' \ar[rr]^\varphi  \ar[rd]_q && E \ar[ld]^p \\
& B
}\]
with $q$ a fiber bundle of compact $n$-manifolds, and $\varphi$ a homotopy equivalence, the stable vertical (i.e.~fiberwise) tangent bundle $(\varphi^{-1})^*T^v E'$ of $q$ defines a map $E\to B\TOP$. A family version of this argument determines a map
\[T^v\colon \calS_\infty(p)\to\map(E,B\TOP).\]

\begin{thm}[({\cite{Hoehn(2009)}})]\label{whtors:Hoehns_result}
If $p$ is a bundle of compact topological manifolds, then the param\-etrized torsion and the vertical tangent bundle induce a weak homotopy equivalence
\[(\tau, T^v)\colon \calS_\infty(p)\to \sections{\Omega\Wh_B(E)}{B} \times \map(E,B\TOP).\]
\end{thm}

In other words, stably, the space of manifold structures on a bundle can be completely described in terms of algebraic $K$-theory and bundle data. Of course, Hoehn works with the parametrized excisive characteristic $\chi^\%$ instead of the torsion. Since these maps differ only by a constant, both statements are equivalent.


\section{The geometric assembly map}\label{fbs}

This section is devoted to the definition and study of the ``geometric assembly map'' on structure spaces. This will lead to the proof of Theorem \ref{intro:main_theorem} which asserts that the geometric assembly map on the level of stable structure spaces has an ``algebraic'' counterpart.

The total space of a bundle of $n$-dimensional compact manifolds over a $k$-dimensional compact manifold is itself an $(n+k)$-dimensional compact manifold. Therefore, given a fibration $p\colon E\to B$, there is a product map
\[\calS_k(B)\times \calS_n(p)\to \calS_{n+k}(E),\]
which on $m$-simplices is defined as follows: If $x\in\calS_k(B)_m$ and $y\in\calS_n(p)_m$ are given by 
\[\xymatrix{
B' \ar[rd] \ar[rr]^\varphi && B\times\Delta^m\ar[ld] && E' \ar[rd] \ar[rr]^\psi && E\times\Delta^m \ar[ld] \\
& {\Delta}^m &&&& B\times\Delta^m
}\]
then the image of $(x,y)$ is given by $\varphi^*E'$, considered as a bundle over $\Delta^m$, together with its canonical map to $E\times\Delta^m$. 

Strictly speaking, of course, this construction again makes use of the chosen bijection $\calU\times\calU\to\calU$.

In particular, if $B$ is itself a $k$-dimensional compact manifold, the identity map on $B$ defines a point in $S_k(B)$; hence we can evaluate the product map to obtain a map
\[\alpha\colon \calS_n(p)\to\calS_{n+k}(E).\]
Notice that the homotopy class of $\alpha$ does not depend on the choice of how we embed $B$ as a subset of $\calU$.

Geometrically $\alpha$ takes all the structures on the fibers of $p$ and assembles them into one big structure. Therefore we are going to call $\alpha$ the \emph{geometric assembly map}. It should not be confused with the homotopy theoretic notion of assembly map as defined in \cite{Weiss-Williams(1995a)}.

\begin{rem}\label{fbs:geometric_and_algebraic_t_over_a_point}
If $B$ is a point, then the map $\alpha\colon \calS_n(E)\to\calS_n(E)$ is canonically homotopic to the identity map. In fact, $\alpha(x)$ and $x$ are canonically homeomorphic. This homeomorphism provides a homotopy.
\end{rem}

As $\alpha$ commutes with stabilization up to a canonical homotopy, there is also a stabilized version
\[\alpha\colon\calS_\infty(p)\to\calS_\infty(E)\]
of the geometric assembly map.

Suppose for simplicity that $B$ is path-connected, and define (after choosing a point $b\in B$):
\[\beta\colon \sections{\Omega\Wh_B(E)}{B}\xrightarrow{\mathrm{Restriction}}\Omega\Wh(F_b)\xrightarrow{\chi_e(B)\cdot j_*} \Omega\Wh(E).\]
Here $j_*$ is induced by the inclusion of the fiber $F_b$ into $E$ and $\chi_e$ denotes the Euler characteristic. The multiplication uses the loop space structure on $\Omega\Wh(E)$. Notice that a different choice of base point $b$ leads to the same map $\beta$, up to homotopy.

\begin{thm}\label{fbs:geometric_and_algebraic_t}
If $p$ is a bundle of compact topological manifolds over a compact connected topological manifold, then the following diagram commutes up to homotopy:
\begin{equation}\label{fbs:diag_geometric_and_algebraic_t}
\xymatrix{
{\calS_\infty}(p) \ar[rr]^\alpha \ar[d]^(0.4)\tau && {\calS_\infty}(E) \ar[d]^\tau \\
{\sections{\Omega\Wh_B(E)}{B}} \ar[rr]^\beta && {\Omega}\Wh(E)
}
\end{equation}
\end{thm}

To prove this result, we are going to show the following two statements:
\begin{enumerate}
\item If $X$ is a compact contractible manifold, then the claim holds for $p$ if and only if the claim holds for the bundle $p\times \id_X$.
\item If $B=B_1\cup_{B_0} B_2$ is a codimension 1 splitting of the base manifold, then the claim holds for $p$ whenever it holds for the restrictions of $p$ onto $B_0, B_1$, and $B_2$.
\end{enumerate}

The homotopy commutativity of \eqref{fbs:diag_geometric_and_algebraic_t} is a formal consequence of these two statements. The proof is parallel to the proof of the general product rule, Theorem \ref{pr:general_product_rule}, and goes by induction on the number of handles in a handlebody decomposition of $B$. Details will be left to the reader.

\begin{proof}[of statement (1)]
Denote by $p'\colon E'\to B'$ the bundle $p\times\id_X$ and let $f\colon B'\to B$ be the projection. Denote by by $\tau'$ the torsion on the structure spaces of $p'$ and $E'$, by $\alpha'$ the geometric assembly map for $p'$ and by $\beta'$ the corresponding map on $K$-theory. We assume that the claim holds for the bundle $p\times\id_X$, i.e.~we assume $\tau'\circ\alpha'\simeq\beta'\circ \tau'$.

Notice that
\begin{enumerate}
\item The composite
\[ \calS_\infty(p)\xrightarrow{f^*} \calS_\infty(p') \xrightarrow{\alpha'} \calS_\infty(E') \xrightarrow{\bar f_*}\calS_\infty(E)\]
agrees with the map $(\times X)\circ \alpha$. 
\item The composite
\[\sections{\Omega\Wh_B(E)}{B} \xrightarrow{f^*} \sections{\Omega\Wh_{B'}(E')}{B'} \xrightarrow{\beta'} \Omega\Wh(E') \xrightarrow{\bar f_*} \Omega\Wh(E)\]
agrees with $\beta$.
\end{enumerate}

Using the product rule and the composition rule (where $\tau(\bar f)=0$), it follows that
\begin{align*}
\tau\circ \alpha & \simeq \tau\circ (\times X)\circ \alpha\\
& \simeq \tau\circ \bar f_*\circ \alpha'\circ f^*\\
& \simeq \bar f_*\circ \tau'\circ \alpha'\circ f^*\\
& \simeq \bar f_*\circ \beta'\circ\tau'\circ f^*\\
& \simeq \bar f_*\circ \beta'\circ f^*\circ \tau\\ 
& \simeq \beta\circ\tau.
\end{align*}

Hence if the claim holds for $p'$, then it holds for $p$. The reverse implication is also true since $f^*$ and $\bar f_*$ are homotopy equivalences.
\end{proof}

\begin{proof}[of statement (2)]
Denote by $j_i\colon B_i\to B$ for $i=0,1,2$ the inclusion. Notice that the codimension 1 splitting of $B$ induces a codimension 1 splitting of $E$ as follows:
\[\xymatrix{
E_0 \ar[r] \ar[d] \ar[rd]^{\bar j_0} & E_2\ar[d]^{\bar j_2}\\
E_1 \ar[r]^{\bar j_1} & E 
}\]
since $p$ is a bundle. Denote by $\square$ this diagram, thought of as a diagram of bundles over the one-point space. Recall the notation $\calS_n(\square)$ from the additivity theorem of the parametrized torsion, and denote by $\gamma_i\colon\calS_n(\square)\to\calS_n(E_i)$ ($i=1,2$), $\gamma_0\colon\calS_n(\square)\to\calS_{n-1}(E_0)$ and $\gamma\colon\calS_n(\square)\to\calS_n(E)$ the forgetful maps. 

The claim follows by considering the following diagram in the homotopy category:
\[\xymatrix{
{\calS_n(p)} \ar[rr]^\alpha \ar@{=}[dd] \ar[rd]_(0.4)\tau && {\calS_{n+k}(E)} \ar[rd]^\tau \\
& {\sections{\Omega\Wh_B(E)}{B}} \ar[dd]_(0.35){\prod_{i=0}^2 j_i^*} \ar[rr]^(0.4)\beta && {\Omega}\Wh(E)\\
{\calS_n(p)} \ar'[r][rr]^(0.35){\hat \alpha} \ar[rd]_(0.4){\prod_{i=0}^2 j_i^*\circ\tau} &&  {\calS_{n+k}(\square)}\ar'[u][uu]^(0.35){\gamma} \ar[rd]_(0.4){\prod_{i=0}^2 \tau\circ\gamma_i}\\
& {\prod_{i=0}^2}\sections{\Omega\Wh_{B_i}(E_i)}{B_i} \ar[rr]^\beta && {\prod_{i=0}^2} \Omega\Wh(E_i) \ar[uu]^(0.7){\bar j_{1*} + \bar j_{2*} - \bar j_{0*}}
}\]
Here $\hat \alpha$ is to denote the obvious map which forgets the base space but remembers the codimension 1 splitting. 

To show that the top square commutes, it is enough to check that all the other squares are commutative.

The commutativity of the left-hand square is purely formal, and the bottom square commutes by naturality of the torsion. The back square is evidently commutative. The commutativity of the front square follows from the definition of $\beta$ and the additivity of the Euler characteristic,
\[\chi_e(B)=\chi_e(B_1)+\chi_e(B_2)-\chi_e(B_0).\]
Finally, the Additivity theorem \ref{strspaces:additivity} guarantees that the right-hand square is commutative.
\end{proof}

For the next result, recall the vertical tangent bundle map
\[T^v\colon \calS(p)\to\map(E,B\TOP)\]
from section \ref{section:sttors}. In the case where $B$ is a point, we write $T:=T^v$ since in this case the vertical tangent bundle is just the tangent bundle.

\begin{prop}\label{fbs:geometric_assembly_on_tangent_bundles}
The diagram
\begin{equation}\label{fbs:diagram_for_tangent_bundles}
   \xymatrix{
{\calS}(p) \ar[rr]^\alpha \ar[d]^{T^v} && {\calS}(E) \ar[d]^{T} \\
{\map}(E,B\TOP) \ar[rr]^{(+p^*TB)} && {\map(E,B\TOP)}
}
\end{equation}
commutes up to homotopy, where the lower horizontal map is given by adding the pull-back of the tangent bundle of $B$.
\end{prop}

Combining diagrams \eqref{fbs:diag_geometric_and_algebraic_t} and \eqref{fbs:diagram_for_tangent_bundles}, we obtain 
\[\xymatrix{
{\calS_\infty}(p) \ar[rr]^\alpha \ar[d]^(0.4){(\tau, T^v)}_(.4)\simeq && {\calS_\infty}(E) \ar[d]^{(\tau,T)}_\simeq \\
{\sections{\Omega\Wh_B(E)}{B}}\times\map(E, B\TOP) \ar[rr]^(.55){\beta\times (+p^*TB)} && {\Omega}\Wh(E)\times\map(E, B\TOP)
}\]
which is the (rotated) diagram appearing in the statement of Theorem \ref{intro:main_theorem}. The vertical maps are weak homotopy equivalences by Hoehn's result, Theorem \ref{whtors:Hoehns_result}. By Theorem \ref{fbs:geometric_and_algebraic_t} and Proposition \ref{fbs:geometric_assembly_on_tangent_bundles}, the diagram is commutative up to homotopy, thus proving Theorem \ref{intro:main_theorem}.

\begin{rem}
The map $(+p^*TB)$ appearing in diagram \eqref{fbs:diagram_for_tangent_bundles} is a homotopy equivalence, since the bundle $TB$ has a stable inverse. It follows that diagram \eqref{fbs:diag_geometric_and_algebraic_t} is a weak homotopy pull-back with respect to any choice of homotopy. The long exact Mayer-Vietoris sequence decomposes into short exact sequences
\[0 \to \pi_*\calS_\infty(p) \xrightarrow{(\tau_*,\alpha_*)} \pi_{*+1}\sections{\Wh_B(E)}{B} \oplus \pi_* \calS_\infty(E) \xrightarrow{\beta_*-\tau_*} \pi_{*+1}\Wh(E)\to 0.\]
\end{rem}

To prove Proposition \ref{fbs:geometric_assembly_on_tangent_bundles}, let us start by giving the explicit construction of $T^v$. Recall the tangent (micro-)bundle \cite{Milnor(1964)} of a topological manifold $E$, which is represented by the diagram
\[E\xrightarrow{\Delta} E\times E\xrightarrow{\Proj_1} E,\]
$\Delta$ denoting the diagonal inclusion and $\Proj_1$ the projection onto the first factor. If $E$ is the total space of a fiber bundle of topological manifolds over $B$, then the vertical tangent bundle is the microbundle over $E$ defined by
\[E\xrightarrow{\Delta} E\times_B E\xrightarrow{\Proj_1} E.\]

By works of Kister \cite{Kister(1964)} and Mazur \cite{Mazur(1964)}, $n$-dimension microbundles form a bundle theory which agrees with the theory of pointed fiber bundles with fiber $\RR^n$. In particular the tangent bundle of any $n$-manifold is represented by a pointed $\RR^n$-bundle. However, there is no canonical choice. To define our map
\[T^v\colon \calS_n(p)\to\map(E,B\TOP(n))\]
let us first define a more suitable model for the classifying space.

Denote by $\mathbf{MB}_n(B)$ the category of all $n$-dimensional microbundles over $B$, considered as a subset of $B\times\calU$, where the morphisms are isomorphism-germs between those in the sense of Milnor \cite[\S 6]{Milnor(1964)}. Pull-back gives rise to a functor
\[\cpCW\op\to\cat,\quad X\mapsto \mathbf{MB}_n(B\times X).\]

\begin{lem}
This functor has the Amalgamation, Straightening, and Fill-in properties as defined in section \ref{strspaces}. 
\end{lem}

Hence, $\opn{MB}_n(E):=\vert N_0\mathbf{MB}_n(E)_\bullet\vert $ is a model for the space $\map(E,B\TOP(n))$. It has the advantage that a microbundle over $E$ canonically defines a point in $\opn{MB}_n(E)$.

\begin{proof}[of Lemma]
We begin with the Amalgamation property. Therefore let 
\[\xymatrix{
B_0 \arincl[rr] \arinclinv[d] && B_1 \arinclinv[d]\\
B_2 \arincl[rr] && B
}\]
be a push-out of compact CW complexes, and let $E_i$, $i=0,1,2$, be microbundles over $B_i$ whose restrictions to $B_0$ agree. Then the push-out of the $E_i$ defines a microbundle over $B$, since each microbundle contains a pointed euclidean bundle, and amalgamation holds for these.

The straightening property is proved in \cite[\S 6]{Milnor(1964)}. To obtain a fill-in, take mapping cylinders whenever we are given actual isomorphisms rather than just isomorphism-germs. In the general case observe that if $E^0\subset E$ is an open neighborhood of the zero-section in a microbundle over $B$, then the open subspace
\[E\times [0,1)\cup E^0\times \{1\}\subset E\times [0,1]\]
is a microbundle over $B\times [0,1]$ which is a fill-in between $E$ and $E^0$. 
\end{proof}

Now, to formally define the map $T^v\colon \calS_n(p)\to\opn{MB}_n(E)$, one first defines a larger structure space $\overline\calS_n(p)$ where a 0-simplex is a diagram
\[\xymatrix{
E' \ar@/^/[rr]^\varphi_{(H)} \ar[rd]_{p'} && E \ar@/^/[ll]^\psi \ar[ld]^p\\
& B
}\]
where $\psi$ is a fiber homotopy inverse of $\varphi$, $H\colon \varphi\circ\psi\simeq \id_E$ is a fiber homotopy,  and $p'$ is a bundle of compact topological manifolds. The forgetful map $\overline\calS_n(p)\to\calS_n(p)$ is a homotopy equivalence; on the other hand the explicit choice of homotopy inverse allows to define a map 
\[\overline\calS_n(p)_\bullet\to \opn{MB}_n(E)_\bullet,\quad (E',\varphi,\psi)\mapsto \psi^*T^v E'\]
where (on the level of 0-simplices) $\psi^* T^v E'$ is a microbundle over $E$, thus defining a 0-simplex in $\opn{MB}_n(E)_\bullet$. Stabilizing yields a map to $\hocolim_n \opn{MB}_n(E)\simeq \map(E,B\TOP)$.

\begin{rem}
Technically speaking, our definition of the tangent microbundle does not allow the manifolds to have boundaries. This problem can be dealt with by adding to all manifolds a collar at the boundary. Existence and uniqueness \cite[Essay I, App.~A]{Kirby-Siebenmann(1977)} of collars shows that the corresponding structure spaces are homotopy equivalent. On the other hand, the collar allows to extend the tangent bundle of the interior of $M$ over the boundary.
\end{rem}

The following lemma shows Proposition \ref{fbs:geometric_assembly_on_tangent_bundles} on the level of 0-simplices. The same technique applies to give a proof on higher simplices -- one just needs to make sure that all the bundles are taken fiberwise over the $\Delta^n$-direction. 

\begin{lem}\label{fbs:decomposition_of_tangent_bundle}
Let $p\colon E\to B$ be a bundle of compact topological manifolds over a compact topological manifold $B$. Let $q\colon U\to B$
be a pointed $\RR^n$-bundle representing the tangent microbundle (as an open subbundle of the projection $B\times B\to B$ onto the first coordinate), with zero-section $s$. Let $\hat U:=(p\times p)^{-1}(U)\subset E\times E$ and $\hat q\colon \hat U\to E$ be the first projection, representing the tangent microbundle of $E$.

Then any choice of fiber homotopy $\id_U\simeq s\circ q$ over $B$ defines a microbundle $X$ over $E\times I$ which restricts to $\hat U$ over $E\times 0$ and to $T_{\fib}E\oplus p^*U$ over $E\times 1$.

In particular, we have
\[TE\cong \hat U\cong T^v E\oplus p^*U\cong T^v E\oplus p^*TB.\]
\end{lem}

\begin{proof}
Denote by $q'\colon V:= p^*U \to E$ the pointed $\RR^n$-bundle pulled back from $q$ using the bundle map $p$. The space $V$ is open in $E\times B$, and the fiber homotopy $\id_U\simeq s\circ q$ induces a fiber homotopy 
\[H\colon V\times [0,1]\to E\times B\]
between the inclusion from $V$ into $E\times B$ and the composite $V\xrightarrow{q'}E\overset{\Delta}{\hookrightarrow} E\times B$.

Now consider the bundle
\[\id\times p\colon E\times E\to E\times B.\]
Pulling it back using $H$, we obtain a bundle $X$ over $V\times [0,1]$ which restricts to  $E\times E\vert_V$ over $V\times 0$; over $V\times 1$, it restricts to $(q')^*(E\times_B E)$.

We may consider the bundle $X$ over $V\times [0,1]$ as a microbundle over $E\times [0,1]$. As such, its restriction over $E\times 1$ is precisely the direct sum $T^v E\oplus p^*U$, and its restriction over $E\times 0$ is the bundle $\hat q$.
\end{proof}


\section{Relation to \texorpdfstring{$h$}{h}-cobordisms}\label{section:h_cob}

Since, for a fiber bundle $p\colon E\to B$ of compact manifolds, the parametrized Whitehead torsion
\[\tau\colon \calS_\infty(p)\to\sections{\Omega\Wh_B(E)}{B}\]
is split surjective, every possible ``higher torsion'', i.e.~any homotopy class of sections in the target of $\tau$ is realized as the parametrized torsion of a fiber homotopy equivalence $p'\to p\times D^k$. The aim of this section is to show that -- under suitable hypotheses -- higher torsions may be realized by $h$-cobordisms and to give an estimate for the stable range.

The strategy of proof is to compare the higher torsion with Waldhausen's map, which connects the space of $h$-cobordisms on a manifold with higher algebraic $K$-theory. The stable parametrized $h$-cobordism theorem \cite{Waldhausen-Jahren-Rognes(2008)} asserts that this latter map is a homotopy equivalence. An explicit comparison between Waldhausen's map and the parametrized excisive characteristic $\chi^\%$ has been used by Hoehn \cite{Hoehn(2009)} to show that $\chi^\%$ is a homotopy equivalence. Our result is a consequence of this comparison and the Bousfield-Kan spectral sequence.

Let us first define the space of parametrized $h$-cobordisms on $p$. Let $\calH(p)_0$ be the set of commutative diagrams
\[\xymatrix{
E' \arincl[rr] \ar[rd]_q && E\times I \ar[ld]^{p\circ\Proj}\\
& B
}\]
where $q$ is a bundle of compact topological manifolds over $B$ (as a subset of $B\times\calU$) such that $E'$ contains a neighborhood of $E\times \{0\}$, and the inclusion of the fibers of $p$ into the fibers of $q$ is a homotopy equivalence. We will think of $E'$ as a fiberwise $h$-cobordism over $E$. The composite of the inclusion $E'\to E\times I$ with the projection onto $E$ defines a retraction from $E$ onto $E'$ which is a homotopy inverse of the inclusion.

\begin{defn}
The \emph{space of $h$-cobordisms on $p$} is the geometric realization $\calH(p)$ of the simplicial set with $k$-simplices $\calH(p)_k:=\calH(p\times\id_{\Delta^k})_0$.
\end{defn}

Denote by $\partial p\colon\partial^v E\to B$ the vertical boundary bundle, whose fiber over $b\in B$ is given by $\partial p^{-1}(b)$. If $E'$ is a fiberwise $h$-cobordism over $\partial^v E$, we can construct the bundle $E'\cup_{\partial^v E} E$ over $B$. The retraction $E'\to \partial^v E$ induces a fiber homotopy equivalence $E'\cup_{\partial^v E} E\to E$. Thus we get a zero-simplex in $\calS_k(p)$, where $k$ is the fiberwise dimension of $p$. Working with families, this rule defines a map
\[\phi\colon \calH(\partial p)\to\calS_k(p)\]
which we may compose with $\tau$ to obtain
\[\calH(\partial p)\xrightarrow{\phi}\calS_k(p)\xrightarrow{\tau}\sections{\Omega\Wh_B(E)}{B}\]

\begin{thm}\label{h_cob:main_result}
Denote by $(F,\partial F)$ the fiber of $p$. If $K$ is in the concordance stable range of $\partial F$, and $(F,\partial F)$ is $l$-connected, then the composite $\tau\circ \phi$ is $(\min(K+1, l-1)-n)$-connected, where $n=\dim B$.
\end{thm}

\begin{cor}
\begin{enumerate}
\item If both $K+1$ and $l-1$ are greater or equal to $\dim B$, then every homotopy class of sections of $\Omega\Wh_B(E)\to B$ is realized as the parametrized torsion of a fiber bundle structure on $p$ obtained by glueing a fiberwise $h$-cobordism along the vertical boundary bundle $\partial p$.
\item If both $K+1$ and $l-1$ are strictly greater than $\dim B$, then there is a one-to-one correspondence between isomorphism classes of fiberwise $h$-cobordisms on the vertical boundary bundle $\partial p$ and homotopy classes of sections of $\Omega\Wh_B(E)\to B$.
\end{enumerate}
\end{cor}

To prove Theorem \ref{h_cob:main_result}, denote by
\[W\colon \calH(N)\to\Omega\Wh(N)\]
Waldhausen's map, for a compact manifold $N$. Hoehn \cite[chapter 3]{Hoehn(2009)} showed

\begin{thm}
The following diagram commutes up to homotopy:
\[\xymatrix{
{\calH(\partial M)} \ar[rr]^\phi \ar[d]^W && {\calS_k(M)} \ar[d]^{\chi^\%} \\
{\Omega\Wh(\partial M)} \ar[rr]^(.4){i_* + \chi^\%(\id_M)} && {\hofib_{\chi(M)} (A^\%(M)\to A(M))}
}\]
\end{thm}

The definition of the parametrized torsion implies that the diagram
\[\xymatrix{
{\calH(\partial M)} \ar[rr]^\phi \ar[d]^W && {\calS_k(M)} \ar[d]^{\tau} \\
{\Omega\Wh(\partial M)} \ar[rr]^{i_*} && {\Omega\Wh(M)}
}\]
commutes up to homotopy as well. 

\begin{proof}[of Theorem \ref{h_cob:main_result}]
Waldhausen's map factors through the infinite stabilization map
\[\calH(\partial M) \to\calH_\infty(\partial M)\to \Omega\Wh(\partial M)\]
where the first map is $(K+1)$-connected by the definition of the concordance stable range, and the last map is a homotopy equivalence by the stable parametrized $h$-cobordism theorem. So $W$ is $(K+1)$-connected.

On the other hand, the inclusion $\partial M\to M$ is $l$-connected by assumption. As $K$-theory preserves connectivity, it follows that the inclusion-induced map $i_*\colon \Omega\Wh(\partial M)\to\Omega\Wh(M)$ is $(l-1)$-connected. Hence the composite $i_*\circ W\simeq \tau\circ\phi$ is $\min(K+1, l-1)$-connected. This proves the result in the case where $B$ is a point. More generally, the result holds for contractible base spaces $B$ as $\tau\circ\phi$ is natural in $B$ and both domain and target are homotopy invariant in $B$.

In the case of a general $B$, use the fact that both the left-hand and the right-hand side of $\tau\circ\phi$ are ``co-excisive'' in $B$, i.e.~given a homotopy push-out $B_3=B_1\cup_{B_0} B_2$ of spaces over $B$, the diagram
\[\xymatrix{
{\calH(\partial p\vert_{B_3})} \ar[rr] \ar[d] &&{\calH(\partial p\vert_{B_2})} \ar[d]\\
{\calH(\partial p\vert_{B_1})} \ar[rr] && {\calH(\partial p\vert_{B_0})}
}\]
is a homotopy pull-back. It follows (see e.g.~\cite[p.~62]{Dwyer-Weiss-Williams(2003)}) that the co-assembly map
\[\calH(p)\to\holim_{\sigma\in\simp B_\bullet}\calH(p\vert_\sigma)\]
is a homotopy equivalence. Thus it is enough to prove the statement for
\[
\holim_{\sigma\in\simp B_\bullet}\calH(p\vert_{\sigma})
\to
\holim_{\sigma\in\simp B_\bullet} \sections{\Omega\Wh_{\vert\sigma\vert}(E_\sigma)}{\vert\sigma\vert}\]
Now it follows from the Bousfield-Kan spectral sequence that this map is $(N-\dim B)$-connected whenever all the maps for the individual $\sigma$ are $N$-connected.
\end{proof}


\appendix

\section{Some results on fibrations}

This appendix collects the technical results on fibrations needed to make our classifying machinery work. Again recall that all the spaces are compactly generated
Hausdorff.

\subsection*{The fibered homotopy extension property}

We begin by stating a useful fact.

\begin{lem}[(Fibered homotopy extension property)]
A map $i\colon A\to X$ is a cofibration if and only if the
fibered homotopy extension property holds:

Given any (solid) commutative diagram 
\[\xymatrix{
A\times I\cup_{A\times 0} X\times 0 \ar[rr]^(.6){h\cup f}
\arinclinv[d] && E
\ar[d]^p\\
X\times I \ar[r]^{\Proj} \ar@{.>}[rru]^H & X\ar[r]^{pf} & B
}\]
with $p$ a fibration, there exists a (dotted) lift.
\end{lem}

In fact, the usual HEP for cofibrations amounts to letting $B$
be a point. On the other hand, if $i$ is a cofibration, the
existence of a lift is a standard result.

As a consequence, some well-known results from the theory of cofibrations also hold in the fibered context. Proofs are identical, replacing the homotopy extension property by the fibered one.

\begin{prop}[({\cite[Chapter 6]{May(1999)}})]\label{fibr:homotopy_equivalences_under_A}
Let 
\[\xymatrix{
 & Y \arinclinv[ld] \arinclinv[rd]\\
X_1 \ar[rr]^f && X_2
}\]
be a diagram of spaces over $B$ (with fiberwise maps), such that the diagonal maps are cofibrations and both $X_1$ and $X_2$ are fibrations over $B$. If $f$ is a fiber homotopy equivalence, then it is a fiber homotopy equivalence relative to $Y$ (i.e.~there is a fiber homotopy inverse and fiber homotopies that fix $Y$).
\end{prop}

Notice that any homotopy equivalence between two fibrations is a fiber homotopy equivalence \cite[chapter 7]{May(1999)}.

\begin{cor} \label{fibr:cofibration_and_h_eq_is_retract}
Suppose that $p\colon E\to B$ is a fibration, and that $i\colon
E'\hookrightarrow E$ is a cofibration. The following are equivalent:
\begin{enumerate}
\item $p\vert_{E'}\colon E'\to B$ is a fibration and $i$ is a homotopy equivalence,
\item $i$ is a fiberwise (strong) deformation retract.
\end{enumerate}
\end{cor}

\begin{prop}[({\cite[Chapter 6]{May(1999)}})]\label{fibr:fiber_homotopy_equivalences_of_pairs}
Let
\[\xymatrix{
Y_1 \ar[rr]^f_\simeq \arinclinv[d] && Y_2\arinclinv[d]\\
X_1 \ar[rr]^g_\simeq && X_2
}\]
be a diagram of spaces over $B$, with vertical maps cofibrations and $X_1$ and $X_2$ fibrations over $B$. If $f$ has a fiber homotopy inverse $f'$ over $B$ and $g$ is a homotopy equivalence, then $(g,f)$ is a fiber homotopy equivalence of pairs.

More precisely, if $H\colon Y_2\times I\to Y_2$ is a homotopy between $\id$ and $f'\circ f$, then there exists a fiber homotopy inverse $g'$ of $g$ that extends $f'$ such that $H$ extends to a homotopy $K\colon X_2\times I\to X_2$ between $\id$ and $g'\circ g$.
\end{prop}

\subsection*{Associated fibration, connections, and regularity}

For $f\colon X\to B$, denote by
$\mathcal E(X)\to B$ the functorially associated fibration, with
$\mathcal E(X)$ being the following pull-back:
\[\xymatrix{
{\mathcal E(X)} \ar[rr] \ar[d]^\gamma  && X \ar[d]^f\\
B^I \ar[rr]^{\eval_0} && B
}\]

\begin{lem}
If $B$ is metrizable, and $p\colon E\to B$ is a fibration, then
the inclusion $E\subset\mathcal E(E)$ is a fiberwise (strong)
deformation retract.
\end{lem}

\begin{proof}
The inclusion $B\to B^I$ of constant functions is a (strong)
deformation retract; so it is a cofibration if there is a map
$\varphi\colon B^I\to I$ such that $B=\varphi^{-1}(0)$. If $d$ is
a metric on $B$, then 
\[\varphi(a)=\sup\bigl( \{d(a(0), a(x)); x\in I \}\cap [0,1]\bigr)\]
is clearly such a map. 

Now, since the inclusion $B\to B^I$ is a section of the fibration
$\eval_0$, it induces a section of the fibration $\mathcal
E(E)\to E$, which then is also a cofibration. The result follows from Corollary \ref{fibr:cofibration_and_h_eq_is_retract}.
\end{proof}

For a map $p\colon E\to B$, a connection is a lift in
the following diagram
\[\xymatrix{
{\mathcal E(E)}\times 0 \ar[rr] \arinclinv[d] && E \ar[d]^p\\
{\mathcal E(E)}\times I \ar[r]^{\gamma\times\id_I}
\ar@{.>}[rru]^C &
B^I\times I
\ar[r]^\eval & B
}\]
Clearly, if $p$ is a fibration, then a connection always exists.
Conversely, the existence of a connection implies that $p$ is a
fibration. Indeed, to give a homotopy lifting problem
\[\xymatrix{
X\times 0 \arinclinv[d] \ar[rr] && E \ar[d]^p\\
X\times I \ar[rr] && B
}\]
is the same thing as to give a map $\alpha\colon X\to \mathcal
E(E)$,
and the composite $C\circ(\alpha\times\id_I)\colon X\times I\to
E$ is then a solution
of the homotopy lifting problem.

\begin{lem}[(Hurewicz)]\label{fibr:regularity}
Any fibration $p\colon E\to B$ over a metrizable space is regular, i.e.~homotopy lifting problems can be solved in such a way that constant paths are lifted to constant paths. More precisely, if $A\subset X$ is a cofibration, and in the homotopy lifting problem
\[\xymatrix{
X\times 0 \cup A\times I \ar[rr] \ar[d] && E \ar[d]^p\\
X\times I \ar@{.>}[rru] \ar[rr]^h && B
}\]
the given lift over $A$ is regular, then there is a lift over $X$ which is regular.
\end{lem}

\begin{proof}
Choose a metric on $B$, and let
\[d\colon B^I\to\RR\]
be the continuous map which assigns to a path its diameter. For $x\in X$, denote by $\gamma_x$ the path in $B$ given by evaluating $h$ on $x$. 

If $H$ is any solution of the homotopy lifting problem above, then 
\[H'(x,t):= H(x, \max(1, \frac{t}{d(\gamma_x)}))\]
is a regular solution. If $H$ is regular on $A$, then $H'$ agrees with $H$ on $A$.
\end{proof}

\subsection*{Fiberwise glueing}

There are two possible ways of glueing fibrations along
cofibrations: Firstly, glueing the fibers and keeping the base
space fixed, and secondly, glueing
fibrations with same fibers over different base spaces. We
begin with the first case.

Here are two useful lemmas that will be freely used.

\begin{lem}[({\cite[Lemma
1.26]{Lueck(1989)}})]\label{fibr:luecks_lemma}
Let 
\[\xymatrix{
A \arinclinv[d]^j \ar[rr]^f && Y \arinclinv[d]^J\\
X\ar[rr]^F && Z
}\]
be a push-out with $j$ a cofibration, and let $p\colon E\to Z$ be
a
fibration.
Then the induced square
\[\xymatrix{
f^*J^* E \ar[rr]^{\bar f} \arinclinv[d]^{\bar j} && J^*E
\arinclinv[d]^{\bar
J}\\
F^* E \ar[rr]^{\bar F} && E
}\]
is a push-out with $\bar j$ a cofibration.
\end{lem}

\begin{lem}\label{fibr:elementary_lemma}
Let $E\to B$ be a map, and let 
\[\xymatrix{
E_0 \ar[rr]^g \arinclinv[d]^i && E_1 \arinclinv[d]\\
E_2 \ar[rr] && E
}\]
be a push-out. Then, for $f\colon A\to B$, the diagram 
\[\xymatrix{
f^*E_0 \ar[rr]^{g'} \ar[d]^{i'} && f^*E_1 \ar[d]\\
f^*E_2 \ar[rr] && f^*E
}\]
is a push-out, too.
\end{lem}

We will see down below that, in the case of fibrations, the vertical maps are actually cofibrations. 

\begin{proof}
The canonical map from the push-out $P:=f^*E_1\cup_{f^* E_0} f^* E_2$ to $f^*E$ is a continuous bijection. To show that it is a homeomorphism, it is enough to show that $P$ carries the subspace topology of $Q:=E_1\times A\cup_{E_0\times A} E_2\times A$ (indeed the natural map $Q\to E\times A$ is a homeomorphism). Therefore we are going to show that the injective continuous map $P\to Q$ is closed.

The inclusion of $f^*E_i$ into $E_i\times A$ is a closed embedding for $i=0,1,2$. Now a subset $Z\subset P$ is closed if and only if it is the image of $\bar Z\subset f^*E_1\coprod f^*E_2$ which is both closed and saturated, i.e. if for $x\in f^*E_0$, we have $i'(x)\in \bar Z$ if and only if $g'(x)\in \bar Z$.

Now $\bar Z$ considered as a subset of $f^*E_1\coprod f^*E_2$ is closed and saturated if and only if it is closed and saturated as a subset of $E_1\times A\coprod E_2\times A$. In this case, the image of $Z$ in $Q$ is closed again, by definition of the topology of $Q$.
\end{proof}

A particularly simple case of glueing arises when all the maps involved are cofibrations.

\begin{lem}\label{fibr:fiberwise_glueing}
Let
\begin{equation}\label{fibr:glueing_diagram}\xymatrix{
E_0 \arincl[rr] \arinclinv[d] && E_1 \arinclinv[d]\\
E_2 \arincl[rr] && E
}\end{equation}
be a push-out square of spaces, with all maps cofibrations, and
let
$p\colon
E\to B$ be a map. If $p\vert_{E_i}\colon E_i\to B$ are fibrations
for
$i=0,1,2$, then so is $p$.
\end{lem}

\begin{proof}
Since $E_0\to E_i$ is a cofibration ($i=1,2$), so is $\mathcal
E(E_0)\to\mathcal
E(E_i)$. So we can find connections $C_i$ for $E_i$ that are
compatible with a
given connection $C_0$ for $E_0$. We thus obtain a connection
\[\singlebox\mathcal E(E)=\mathcal E(E_1)\cup_{\mathcal E(E_0)} \mathcal
E(E_2)
\to E=
E_1\cup_{E_0} E_2.\esinglebox\]
\end{proof}

\begin{cor}\label{fibr:cofibration_preserved_under_pullback}
Let $f\colon E'\to E$ be a fiberwise map between fibrations over
$B$.
If $f$ is
a cofibration, then
\begin{enumerate}
 \item the canonical map from the mapping
cylinder $\cyl(f)$ to $B$ is a fibration, and
\item for any map $g\colon A\to B$, the induced map $g^*f\colon
g^*E'\to g^*E$
is a cofibration.
\end{enumerate}
\end{cor}

So, if a fiberwise map between fibrations is a cofibration on the level of total spaces, then it is also a cofibration on each fiber. The converse statement does not hold, as can be easily seen using Tulley's construction (described below).

\begin{proof}
(i) follows directly from Lemma \ref{fibr:fiberwise_glueing}.

(ii) The map $f$ being a cofibration is equivalent to saying the
the canonical map $\Cyl(f)\to E\times I$ is a cofibration
and a homotopy equivalence. So, as $\Cyl(f)$ is a fibration over
$B$, Corollary \ref{fibr:cofibration_and_h_eq_is_retract} implies
that the inclusion is a fiberwise deformation retract. Therefore
its restriction, namely the inclusion $\Cyl(g^*f)\to g^*E\times
I$, is still a retract. 
\end{proof}

If we relax the cofibration condition somewhat, we have to put stronger hypotheses on the base space. See section \ref{strspaces} for the definition of the property ULC.

\begin{thm}[({\cite[Thm.~2.5]{ArnoldJ(1973)}})]\label{fibr:fiberwise_glueing_2}
Lemma \ref{fibr:fiberwise_glueing} still holds if we only assume that the vertical maps in Diagram \eqref{fibr:glueing_diagram} are cofibrations, provided that $B$ is metrizable ULC. 

In particular, the mapping cylinder of a fiberwise map between fibrations over such spaces $B$ is again a fibration over $B$.
\end{thm}

\subsection*{Glueing over different base spaces}

Recall the definition of ULC from section \ref{strspaces}.

\begin{prop}[(compare {\cite[Thm.~4.2]{ArnoldJ(1972)}})]
\label{fibr:glueing_over_base_spaces}
Let $X$ be metrizable, let
\[\xymatrix{
X_0 \arincl[rr] \arinclinv[d] && X_1 \arinclinv[d]\\
X_2 \arincl[rr] && X
}\]
be a push-out square with all maps cofibrations, and $X_i$ ULC
($i=1,2$). If $B$ is another metrizable
space, and $p\colon E\to B\times X$ is a map such
that the restriction of $p$ to $p^{-1}(B\times X_i)\to B\times
X_i$ is a fibration ($i=1,2$), then $p$ is a fibration.
\end{prop}

\begin{proof}[of Proposition \ref{fibr:glueing_over_base_spaces}]
We are going to show that for each $x\in X_0$ there is a
neighborhood
$U$ of
$x$ in $X$ such that $p$, restricted over $B\times U$, is a
fibration.

The proof uses the observation that any fiberwise retract of a
fibration
is a
fibration. In our case, write $p=(p_1,p_2)$, and let, for some
$U\subset X$, 
$q=p_1\times\id_U\colon E\vert_{B\times U}\times U\to B\times U$.
The
map $q$ is
a fibration, and there is a commutative diagram 
\[\xymatrix{E\vert_{B\times U} \ar[rd]^{p} \ar[rr]^{\id\times
p_2}
&&
E\vert_{B\times U}\times U \ar[ld]_q \\
& B\times U
}\]
So it is enough to construct a fiberwise map (often called ``slicing function'') 
\[\varphi\colon E\vert_{B\times U}\times U\to E\vert_{B\times U}\]
such that $\varphi\circ(\id\times p_2)=\id$.

The hypotheses on the base spaces imply (see e.g.~\cite[p.~182]{ArnoldJ(1972)})
that
there is a neighborhood $U$ of $x$ and a homotopy
\[\sigma\colon U \times U\times I\to X\]
between the projection onto the first and the projection onto the
second
factor, such that the homotopy is stationary on the diagonal and,
for
each
$(u,u')\in U\times U$, the corresponding path $\alpha$ from $u$ to
$u'$ has the
following additional property:

For each $i,j$, if $\alpha(0)\in X_i$, then
$\alpha([0,\frac12])\subset X_i$, and if $\alpha(1)\in X_j$, then
$\alpha([\frac12,1])\subset X_j$. 

Consider now the following homotopy lifting problem:
\[\xymatrix{
E\vert_{B\times U}\times U \times 0 \arinclinv[d] \ar[rr]^{\Proj_E} &&
E\ar[d]^{p}\\
E\vert_{B\times U}\times U\times I
\ar[r]^{p\times\id}\ar@{.>}[rru]
&B\times U\times  U\times I 
\ar[r]^(.65){\id\times\sigma} & B\times X
}\]
This particular problem can be solved at least on
$E\vert_{B\times U}\times U\times [0,\frac12]$ (by choosing compatible regular connections for the fibrations over $X_i$, $i=0,1,2$). Then we can solve the problem similarly on
$E\vert_{B\times U}\times U\times
[\frac12,1]$, thus obtaining a solution of the whole problem. 

As both $B$ and $X$ are metrizable, there is a regular solution by Lemma \ref{fibr:regularity}. Such a solution of our problem, evaluated at 1, is the map $\varphi$ we need.
\end{proof}

Next we deal with the question whether this glueing procedure
preserves cofibrations. Here is a preparatory lemma.

\begin{lem}[(Compare {\cite[Thm.~13]{Strom(1968)}})]\label{fibr:cofibration_condition_on_pairs}
Let $i\colon A\hookrightarrow B$ be a cofibration, $p\colon X\to B$ a
fibration,
$j\colon Y\hookrightarrow X$ a cofibration such that $p\vert_{Y}\colon
Y\to B$ is a fibration. Then the inclusion
\[X\vert_A\cup Y\hookrightarrow X\]
is a cofibration.
\end{lem}

\begin{proof}
By Str\o ms version of the NDR property (see \cite[I.5.14]{Whitehead(1978)}), the inclusion
of $A$
into $B$
being a cofibration is equivalent to the existence of $H\colon
B\times
I\to B$
and $\varphi\colon B\to I$ such that
\begin{enumerate}
\item $A\subset \varphi^{-1}(0)$,
\item $H$ is stationary on $A$,
\item $H_0=\id_B$,
\item $H(b,t)\in A$ whenever $\varphi(b)<t$.
\end{enumerate}

In fact, any retraction $r\colon B\times I\to A\times I\cup B\times 0$ gives rise to
\[H(x,t):=\Proj_B r(x,t)\quad \mathrm{and}\quad \varphi(x):=\sup_{t\in I} (t-\Proj_I r(x,t)).\]

As $Y\hookrightarrow X$ is also a cofibration, let $D\colon
X\times I\to X$, $\psi\colon X\to I$ such that the analogues of conditions (i) to (iv) hold. It follows from Corollary \ref{fibr:cofibration_preserved_under_pullback} that $Y\times I\cup X\times 0$ is a fiberwise retract of $Y\times I$, hence we can even assume that $D$ is also a fiber homotopy over $B$.

Let 
\begin{align*}
\eta\colon & X\to I,  & \eta(x)&=\min(\psi(x),\varphi p(x)),\\
G\colon & X\times I\to X, & G(x,t)&=\bar H[D(x,\min(t,\varphi
p(d))),
\min(t,\eta(x))],
\end{align*}
where $\bar H\colon X\times I\to X$ is a lift of $H$ such that
$\bar
H(Y\times
I)\subset Y$ and $\bar H_0=\id_X$.

Then,
\begin{enumerate}
\item $X\vert_A\cup Y \subset\eta^{-1}(0)$,
\item $G$ is stationary on $X\vert_A\cup Y$,
\item $G_0=\id_X$,
\item if $\eta(x)<t$, then either 
 \begin{enumerate}
  \item $\psi(x)<\varphi p(x)$, in which case $\min(t,\varphi
p(x))>\psi(x)$,
so 
\[D(x,\min(t,\varphi p(x)))\in Y\]
and
\[G(x,t)=\bar H[D(x,\min(t,\varphi p(x))),\min(t,\eta(x))]\in
Y,\]
  \item or $\varphi p(x)\leq\psi(x)$ and $\varphi p(x)<1$, in which case
$\min(t,\eta(x))=\varphi p(x)$, so 
\begin{align*}
\multbox
pG(x,t) & = p\bar H[D(x,\min(t,\varphi p(x))),\varphi p(x)]\\
& =H[pD(x,\min(t,\varphi p(x))),\varphi p(x)]\\
& = H[p(x), \varphi p(x)]\in A. 
\emultbox
\end{align*}
\end{enumerate}
\end{enumerate}
\end{proof}

\begin{prop}\label{fibr:glueing_and_cofibrations}
Let $p\colon E\to B$ be a fibration, $E'\subset E$ a 
subspace such that $p\vert_{E'}\colon E'\to B$ is a fibration.
Suppose that $B$ is a push-out
\[\xymatrix{
B_0 \arincl[rr] \arinclinv[d] && B_1 \arinclinv[d]\\
B_2 \arincl[rr] && B
}\]
with all maps cofibrations. Let $E_i:=E\vert_{B_i}$
and $E'_i:=E'\vert_{B_i}$. If the inclusions $j_i\colon E'_i\to
E_i$ are cofibrations for $i=0,1,2$, then so is the inclusion
$j \colon E'\to E$. 
\end{prop}

This proposition follows at once from Lemma \ref{fibr:cofibration_condition_on_pairs} together with the following Lemma whose proof is an exercise using the definitions.

\begin{lem}\label{fibr:glueing_cofibrations}
Let 
\[\xymatrix{
X_1\arinclinv[d] && X_0\arinclinv[d] \arinclinv[ll] \arincl[rr] && X_2 \arinclinv[d]\\
Y_1 && Y_0 \arinclinv[ll] \arincl[rr] && Y_2
}\]
be a diagram of spaces with all maps cofibrations. If the induced maps
\[X_1\cup_{X_0} Y_0\to Y_1\quad\mathrm{and}\quad X_2\cup_{X_0} Y_0\to Y_2\]
are cofibrations, then so is
\[X_1\cup_{X_0} X_2 \to Y_1\cup_{Y_0} Y_2.\]
\end{lem}

\subsection*{Tulley's construction}

\begin{prop}[({\cite{Tulley(1969)}})]\label{fibr:tulleys_result}
Let $p\colon E\to B$ be a fibration over a metrizable space $B$,
and
let
$i\colon E'\to E$ a fiberwise (strong) deformation retract.
Denote 
\[T(i):=E'\times \{0\}\cup E\times (0,1]\subset E\times I.\] 
Then $p\vert_{E'}\colon E'\to B$ is a
fibration, and so is the obvious map $q\colon T(i)\to B\times I$
induced by $p$
and $i$. 
\end{prop}

\begin{proof}[of Proposition \ref{fibr:tulleys_result}]
Denote by $\pi\colon T(i)\to E$ the obvious projection. 
Let $K\colon E\times I\to E$ be a fiber homotopy between $\id_E$ and a
retraction onto $E'$. Let the left square in the following
diagram be a homotopy lifting problem that we wish to solve.
(Here $J=I=[0,1]$.)
\[\xymatrix{
X\times 0 \arinclinv[d] \ar[rr]^{f} && T(i)
\ar[d]^q \ar[r]^{\pi} & E \ar[d]^p\\
X\times J \ar[rr]^{h=(h_1,h_2)} \ar@{.>}[rru]^{H} && B\times I
\ar[r]^{\Proj}& B
}\]
Using the outer square and the fact that $p$ is a fibration, we
obtain a lift $L\colon X\times J\to E$ of $h_1$ that extends the
$\pi\circ f$. Now, let
$H(x,t):=(H_1(x,t), h_2(x,t))\in E\times I$ with \[H_1(x,t)=
\begin{cases}
K\bigl(L(x,t), \min(1,\frac{t}{h_2(x,t)})\bigr), &\mathrm{if}\;
t>0,\\
\pi\circ f(x) &\mathrm{if}\; t=0 
\end{cases}
\]
(letting $\min(1,\frac{t}{0})=1$).

We have to check continuity in $t=0$. So let $(x_\alpha,
t_\alpha)$ be a net converging to $(x,0)$. The only nontrivial
case is when $h_2(x_\alpha,t_\alpha)$ tends to 0, too. In
this case $L(x_\alpha,t_\alpha)$ converges to $\pi\circ f(x)$ which lies
in $E'$. But if $e_\alpha$ converges to some element $e\in E'$,
then $K(e_\alpha, \tau_\alpha)$ converges to $e$ for any net
$(\tau_\alpha)$ on $I$, as $K$ is stationary on $E'$ and the unit
interval is compact.

Now, whenever $h_2(x,t)=0$, we have $H_1(x,t)\in E'$, thus $H$
really defines a map to $T(i)$. Moreover $p\circ
H_1(x,t)=h_1(x,t)$, as $K$ is a fiber homotopy and $L$ is a lift
of $h_1$, so $q\circ H=h$. 
\end{proof}

\begin{lem}\label{fibr:tulleys_construction_and_cofibrations}
Suppose that we have a square of fibrations
\[\xymatrix{
E'_0 \arincl[r]^{i_0}_\simeq \arinclinv[d]^{j'} & E_0\arinclinv[d]^j\\
E'_1 \arincl[r]^{i_1}_\simeq & E_1
}\]
over $B$, with all maps cofibrations and the horizontal maps fiber homotopy equivalences. If the canonical map $E'_1\cup_{E'_0} E_0\to E_1$ is a cofibration, then so is the inclusion $T(i_0)\to T(i_1)$.
\end{lem}

\begin{proof}
Recall that the cofibration condition implies that there are retractions $r\colon E_1\times I\to\Cyl j$ and $r'\colon E'_1\times I\to\Cyl j'$. The additional condition is easily seen to guarantee that $r$ can be modified in such a way that $r$ and $r'$ are compatible, i.e.~$r\vert_{E'_1\times I}=r'\colon E'_1\times I\to\Cyl j$.

Now recall from the proof of Lemma \ref{fibr:cofibration_condition_on_pairs} how to get a NDR pair structure $(\varphi, H)$ on $(E_1,E_0)$ from  the retraction $r$.

Since $r$ and $r'$ are compatible, $(\varphi,H)$ and the corresponding maps $(\varphi', H')$ for the NDR pair $(E'_1, E'_0)$ are compatible. Hence we can define $\bar\varphi\colon T(i_1)\to I$
by applying either $\varphi$ or $\varphi'$ to the first coordinate of $(x,t)\in T(i_1)\subset E_0\times I$. Define $\bar H\colon T(i_1)\times I\to T(i_1)$ in a similar way. The maps $\bar\varphi$ and $\bar H$ satisfy the properties required for  $(T(i_1), T(i_0))$ to be an NDR pair.
\end{proof}

\subsection*{Fill-in for fibrations}

The goal of this section is to show that fill-ins exist for
fibrations and for fibrations with fiberwise splitting (as considered in section \ref{whtors}). We will see that Tulley's construction, together with Theorem \ref{fibr:fiberwise_glueing_2}, will produce fill-ins. To make sure that we have enough space to deal with homotopies, we redefine now $T(i):=E'\times [0,\frac12]\cup E\times (\frac12,1]$. It can be understood as a glueing of the original construction with the trivial fibration $E'\times I$. So the essential properties remain unchanged.

\begin{prop}\label{fibr:fill_in}
The functor $\mathbf{Fib}(B;F)$ from section
\ref{strspaces} satisfies the Fill-in Condition if $B$ is metrizable ULC.
\end{prop}

For the reader's convenience, and to establish notation, we recall the Fill-in property for our special case:

For any three fibrations $E', E'', E$ over $B\times \Delta^k$ and
any two fiber homotopy equivalences $\varphi'\colon E'\to E$ and $\varphi''\colon E''\to E$, the following holds: There exists a fibration $\bar E$ 
over $B\times \Delta^k\times I$ and a fiber homotopy
equivalence $\Phi\colon \bar E\to E\times I$ which restricts
to $\varphi'$ over 0 and to $\varphi''$ over 1. 

Moreover, given two more fibrations $F', F''$ over $B\times \Delta^k$
and fiber homotopy equivalences $\psi'\colon F'\to E$ and $\psi''\colon F''\to E$,
which agree with $(E', E'', \varphi', \varphi'')$ when
restricted to a collection of faces of $B\times \Delta^k$, there are
extensions $(\bar E, \Phi)$ and $(\bar F, \Psi)$ of $(E', E'',
\varphi', \varphi'')$ and $(F', F'', \psi', \psi'')$ that
agree when restricted to the same collection $\times I$.

\begin{proof}[of Proposition \ref{fibr:fill_in}]
Let us first consider a special case: Suppose that the
map $(\varphi'')^{-1}\circ\varphi'\colon E'\to E''$ is fiber
homotopic to a map $\alpha$ which is a cofibration and homotopy equivalence. In this case the
canonical map $T(\alpha)\to B$ is a fibration by Proposition
\ref{fibr:tulleys_result}, and it is easy to construct a map
$\Phi\colon T(\alpha)\to E\times I$ as requested, using a fiber homotopy
$\varphi''\circ\alpha\simeq \varphi'$. 

In the general case, the mapping cylinder $\Cyl(\alpha)$ is a fibration over
$B\times \Delta^k$ by Theorem \ref{fibr:fiberwise_glueing_2}. It comes with a canonical map $\hat\varphi\colon
\Cyl(\alpha)\to E''\to E$. In this case both inclusions of $E'$
and $E''$ into the mapping cylinder are cofibrations and homotopy
equivalences. So, by the previous case, fill-ins exist between
$(E',\varphi')$ and $(\Cyl(\alpha),\hat\varphi)$, and between
$(\Cyl(\alpha),\hat\varphi)$ and $(E'',\varphi'')$. Glueing these
(using the Amalgamation property), one now obtains a fill-in
between $(E',\varphi')$ and $(E'',\varphi'')$.

Proposition \ref{fibr:fiber_homotopy_equivalences_of_pairs} guarantees that this can be chosen to be compatible over smaller simplices.
\end{proof}

The interested reader may want to compare this proof with the intricate argument from \cite[Lemma 17.3]{Hughes-Taylor-Williams(1990)}. --- The same argument, together with Lemma \ref{fibr:tulleys_construction_and_cofibrations} shows:

\begin{prop}\label{fibr:fill_in_with_splitting}
The functor $\mathbf{Fib}(B;(F_i))$ from section
\ref{whtors} satisfies the Fill-in Condition if $B$ is metrizable ULC.
\end{prop}


\section{Comparison with the unparametrized case}\label{section:unparametrized_case}

The aim of this section is to give an elementary argument that the param\-etrized Whitehead torsion reduces to the classical Whitehead torsion in the unparam\-etrized setting. It is the linking part between classical simple homotopy theory and higher algebraic $K$-theory.

\begin{prop}
Let $M$ be a compact topological $n$-manifold. Then, the map
\[\tau\colon\pi_0\calS_n(M)\to\pi_0\Omega\Wh(M)\cong\Wh(\pi_1 M)\]
sends $f\colon N\to M$ to the Whitehead torsion of $f$. 
\end{prop}

To prove this proposition, we need to compare the models for $A$-theory we use with Waldhausen's ones. We will make use of Waldhausen's notation from \cite{Waldhausen(1985)} without further remarks. In this notation, Waldhausen's models are
\begin{align*}
A^\%(X) & = \Omega\vert sS_\bullet\mathcal R_f(X^{\Delta^\cdot})\vert,\\
A(X) & = \Omega\vert wS_\bullet\mathcal R_f(X^{\Delta^\cdot})\vert,\\
\Omega\Wh(X)&  = \vert s\mathcal C^h_f (X)\vert
\end{align*}
The assembly map $A^\%(X)\to A(X)$ is induced by the inclusion of categories $s\mathcal R_f(X)\to w\mathcal R_f(X)$. 

\begin{rem}
\begin{enumerate}
\item This model for $A(X)$ has $\pi_0 A(X)=\ZZ$, whereas the one used by Dwyer-Weiss-Williams has $K_0(\ZZ[\pi(X)])$ as path components. But since here all spaces are homotopy finite, the path component of their characteristic, which is the Euler characteristic, lies in $\ZZ$, so the homotopy fiber
\[\hofib_{\chi(M)}(A^\%(M)\to A(M))\]
is the same. 
\item Waldhausen's models presented here are functors from simplicial sets and not from spaces. So we assume all manifolds to be triangulated. Strictly speaking, we have to replace the manifold $M$ by a triangulated closed disc bundle $A$ over $M$, which is always possible by \cite[Essay 3]{Kirby-Siebenmann(1977)}. 
\end{enumerate}
\end{rem}

Consider the composites
\begin{align*}
\vert s\mathcal R_f(X)\vert & \to \Omega\vert sS_\bullet\mathcal R_f(X)\vert \to \Omega\vert sS_\bullet\mathcal R_f(X^{\Delta^\cdot})\vert=A^\%(X)\\
\vert w\mathcal R_f(X)\vert & \to \Omega\vert wS_\bullet\mathcal R_f(X)\vert \to \Omega\vert wS_\bullet\mathcal R_f(X^{\Delta^\cdot})\vert=A(X)
\end{align*}
The first maps in both composites are the maps reminiscent of the group completion, whereas the second maps come from the inclusion $A_0\to A_\bullet$ which exist for any simplicial space $A_\bullet$. For any compact manifold $M$, the object $M\amalg M$ of $\mathcal R_f(M)$ (with obvious section and retraction) defines an element $\chi^\%(M)\in A^\%(M)$ and an element $\chi(M)\in A(M)$. Since $\alpha\chi^\%(M)=\chi(M)$, we even obtain an element 
\[\chi^\%(M)\in\pi_0\hofib_{\chi(M)}(A^\%(M)\to A(M)).\]
We are going to call this element Waldhausen's excisive characteristic. (This characteristic is \emph{not} excisive in the sense of \cite{Dwyer-Weiss-Williams(2003)}, but this difference is not important here. See the introduction of \cite{Weiss(2002)} for a more detailed discussion.)

Notice that the homotopy invariant characteristic $\chi$ defined here is the same as the one of Dwyer-Weiss-Williams, modulo the homotopy equivalence 
\[\Omega\vert wS_\bullet\mathcal R_f(X)\vert \to \Omega\vert wS_\bullet\mathcal R_f(X^{\Delta^\bullet})\vert.\]
Not surprisingly, the excisive characteristics agree as well. This is a consequence of the following two observations:
\begin{enumerate}
\item Both definitions of the excisive characteristic agree certainly if $M$ is contractible, since in this case the homotopy fiber is contractible.
\item Moreover both definitions are additive in the sense that if $M=M_1\cup_{M_0} M_2$ is a codimension one splitting, then 
\[\chi^\%(M) = j_{1*} \chi^\%(M_1) + j_{2*} \chi^\%(M_2) - j_{0*} \chi^\%(M_0),\]
with $j_i$ the obvious inclusions. The proof in the Waldhausen setting is basically identical to the proof for the homotopy invariant characteristic; both are a consequence of Waldhausen's additivity theorem.
\end{enumerate}

Now, given a homotopy equivalence $\phi\colon N\to M$ between compact manifolds, the path component of the element
\[\tau(\phi)=\chi^\%(M)-\phi_*\chi^\%(N)\in\pi_0\hofib_*(A^\%(M)\to A(M))\]
is represented by the object $\cone(\phi)$, considered as a retractive space over $M$ in the obvious way. This is seen using the cofibration sequence in $\mathcal R_f(M)$
\[M\amalg M \rightarrowtail \cone(\phi) \twoheadrightarrow \Sigma(M\amalg N),\]
using that suspension induces $-\id$ on $K$-theory.

Now we only need to observe that the space $\cone(\phi)$, as a space relative to $M$, also defines an object of the category $\mathcal C^h_f(M)$ and therefore an element in $\Omega\Wh(M)=\vert s\mathcal C^h_f(M)\vert$. It represents $\tau(\phi)$ under the isomorphism $\pi_0\Omega\Wh(M)\cong \pi_0\hofib_*(A^\%(M)\to A(M))$. 

To complete the proof, observe that $\pi_0\Omega\Wh(M)$ is just the geometric definition of the Whitehead group, such that under the isomorphism $\pi_0\Omega\Wh(M)\cong\Whp{M}$ (see e.g.~\cite{Cohen(1973)}), the element $\tau(\phi)$ corresponds to the Whitehead torsion of $\phi$. 


\affiliationone{
   Wolfgang Steimle\\
   Mathematisches Institut\\
   Endenicher Allee 60\\
   53115 Bonn, Germany
   \email{steimle@math.uni-bonn.de}
}

\end{document}